\documentclass[a4,11pt]{amsart}
\usepackage{etex}
\usepackage{amssymb}
\usepackage{amstext}
\usepackage{amsmath}
\usepackage{amscd}
\usepackage{amsthm}
\usepackage{amsfonts}
\usepackage{enumerate}
\usepackage{graphicx}
\usepackage{epic}
\usepackage{latexsym}
\usepackage[all]{xy}
\usepackage[usenames]{color}
\usepackage[pdftex]{hyperref}
\hypersetup{
  bookmarksnumbered = true,
  bookmarksopen = false,
  bookmarksopenlevel = 2,
  colorlinks = false,
  pageanchor=true,
  pdfcreator={Takahide Adachi, Takuma Aihara, Aaron Chan},
  pdfstartview = FitW
}

\newtheorem*{corollary*}{Corollary}
\newtheorem{theorem}{Theorem}[section]

\newtheorem{corollary}[theorem]{Corollary}
\newtheorem{lemma}[theorem]{Lemma}
\newtheorem{proposition}[theorem]{Proposition}
\newtheorem{question}[theorem]{Question}

\newtheorem*{claim*}{Claim}

\theoremstyle{definition}
\newtheorem{definition}[theorem]{Definition}
\newtheorem{remark}[theorem]{Remark}
\newtheorem{example}[theorem]{Example}

\newtheorem{definitiontheorem}[theorem]{Definition-Theorem}

\theoremstyle{remark}

\numberwithin{equation}{theorem}
\renewcommand{\mod}{\operatorname{mod}}
\newcommand{\proj}{\operatorname{proj}}
\newcommand{\End}{\operatorname{End}}
\newcommand{\Hom}{\operatorname{Hom}}
\newcommand{\add}{\operatorname{add}}

\newcommand{\soc}{\operatorname{\mathsf{soc}}}

\newcommand{\Kb}{\mathsf{K}^{\rm b}}
\newcommand{\T}{\mathcal{T}}

\newcommand{\CC}{\mathcal{C}}

\newcommand{\Z}{\mathbb{Z}}

\newcommand{\ol}{\overline}
\newcommand{\xto}{\xrightarrow}
\newcommand{\from}{\operatorname{\leftarrow}}

\newcommand{\tilt}{\operatorname{\mathsf{tilt}}}
\newcommand{\ipt}{\operatorname{\mathsf{2ipt}}}
\newcommand{\gcx}{\operatorname{\mathsf{2scx}}}
\newcommand{\SW}{\operatorname{\mathsf{SW}}}
\newcommand{\AW}{\operatorname{\mathsf{AW}}}
\newcommand{\CW}{\operatorname{\mathsf{CW}}}
\newcommand{\ntilt}[1]{\operatorname{\mathsf{#1-tilt}}}
\newcommand{\ttilt}{\operatorname{\mathsf{2-tilt}}}
\newcommand{\e}{\epsilon}
\newcommand{\mm}{\operatorname{\mathfrak{m}}}

\newcommand{\val}{\mathrm{val}}\newcommand{\vir}{\mathrm{vr}}
\newcommand{\op}{\mathsf{op}}
\newcommand{\G}{\mathbb{G}}
\newcommand{\W}{\mathbb{W}}
\newcommand{\hd}{\mathfrak{h}}
\newcommand{\tail}{\mathfrak{t}}
\newcommand{\word}{\mathsf{w}}

\def\addsym #1: #2#3{#1 \> \parbox{4.8in}{#2 \dotfill \pageref{#3}}\\}
\makeindex

\begin{document}
\title[Two-term tilting complexes over Brauer graphs algebras]{Classification of two-term tilting complexes over Brauer graph algebras}
\date{\today}
\author{Takahide Adachi}
\address{Graduate school of Science, Osaka Prefecture University, 1-1 Gakuen-cho, Nakaku, Sakai, Osaka, 599-8531, Japan}
\email{adachi@mi.s.osakafu-u.ac.jp}
\author{Takuma Aihara}
\address{Department of Mathematics, Tokyo Gakugei University,
4-1-1 Nukuikita-machi, Koganei, Tokyo 184-8501, Japan}
\email{aihara@u-gakugei.ac.jp}
\author{Aaron Chan}
\address{Graduate School of Mathematics, Nagoya University, Furocho, Chikusaku, Nagoya 464-8602, Japan}
\email{aaron.kychan@gmail.com}
\thanks{2010 {\em Mathematics Subject Classification.} 16G10;  16G20, 18E30}
\thanks{{\em Key words and phrases.} tilting complex, Brauer graph algebra, ribbon graph}
\thanks{T. Aihara and A. Chan was partly supported by IAR Research Project, Institute for Advanced Research, Nagoya University.
T. Aihara was supported by Grant-in-Aid for Young Scientists 15K17516.
T. Adachi is supported by Grant-in-Aid for JSPS Research Fellow 17J05537.
A. Chan is supported by JSPS International Research Fellowship.
}

\begin{abstract}
Using only the combinatorics of its defining ribbon graph, we classify the two-term tilting complexes, as well as their indecomposable summands, of a Brauer graph algebra.  As an application, we determine precisely the class of Brauer graph algebras which are tilting-discrete.
\end{abstract}

\maketitle

\section{Introduction}
The derived category of a non-semisimple symmetric algebra is a beast.  For example, a simple classification of perfect objects usually does not exist, which makes the task of finding derived equivalent algebras - one of the central themes in modular representation theory of finite groups, extremely difficult.  The Okuyama-Rickard construction \cite{Oku,Ric1} gives an easy way to calculate non-trivial tilting complexes.  Therefore, one would naturally hope that all tilting complexes can be obtained by applying such a  construction repeatedly, and  use this to determine, say, the derived equivalence class.  Roughly speaking, an algebra is tilting-connected if such a hope can be realised.

One of the main results in \cite{AI} is that there is a partial order structure on the set of tilting complexes.  Using this translation, tilting-connectedness simply means that the Hasse quiver of this partially ordered set is connected.  One could attempt to exploit this combinatorial property to better understand, for instance, the derived Picard group(oid) of the derived category.  See \cite{Zvo2} for a fruitful result in this direction.

Determining tilting-connectedness is not easy - at least with the current technology.  So far, the only known tilting-connected symmetric algebras are the local ones \cite{AI}, and the representation-finite ones \cite{A}.  The ``proof" for tilting-connectedness of local symmetric algebras in fact is the answer we wanted originally - a classification of tilting complexes.  More precisely, a tilting complex of a local algebra is precisely a stalk complex given by a finite direct sum of the (unique) indecomposable projective module.  On the other hand, the proof for the representation-finite case does not generalise to an arbitrary finite dimensional symmetric algebra.

A symmetric algebra is said to be tilting-discrete if the set of $n$-term tilting complexes is finite for any natural number $n$.  In such a case, the algebra will be tilting-connected.  This notion is introduced in \cite{A} in order to find a more computable approach to determine tilting-connectedness.  To see if this approach can be successful for non-representation-finite non-local symmetric algebras, we use the Brauer graph algebras as a testing ground.  We remark here that the tilting-discreteness, along with its applications, of the preprojective algebras is also studied in a parallel work of the second author and Mizuno \cite{AM}.

The representation-finite Brauer graph algebras (the Brauer tree algebras) were discovered by Brauer in the forties during the dawn of modular representations of finite groups.
In this article, we focus on their generalisation called the Brauer graph algebras.
These algebras are tame symmetric algebras which are arguably the easiest class of symmetric algebras that one can play with.  This is because the structure of such an algebra is encoded entirely in a simple combinatorial object, called the Brauer graph.  It often turns out that one can replace many algebraic and homological calculations into simple combinatorial games on the Brauer graph.  These results in turn would inspire further development in the techniques and theories for larger classes of algebras, such as group algebras or tame symmetric algebras.

The first homological calculation which will be turned into pure combinatorics is the determination of two-term (pre)tilting complexes (Theorem \ref{mainbij}).  Since the combinatorics is entirely new, we avoid giving the statement of the theorem here.
For readers with no knowledge about Brauer graphs, we briefly 
recall that a Brauer graph is a graph with a cyclic ordering of the edges around each vertex, and a positive number called multiplicity associated to each vertex.  Forgetting the multiplicities, one obtains an orientable ribbon graph (or fatgraph) - a combinatorial object which appears in many other areas of mathematics, such as dessin d'enfants, Teichm\"{u}ller and moduli space of curves, homological mirror symmetry, etc.

Our second result is the classification of tilting-discrete Brauer graph algebras.  We achieve this by using the ``two-term combinatorics" and the tilting mutation theory for Brauer graph algebras.  Since most of the tilting-discrete Brauer graph algebras are neither representation-finite nor local, we have obtained new classes of tilting-connected symmetric algebras.

\begin{theorem}
Let $\G$ be a Brauer graph, and $\Lambda_\G$ its associated Brauer graph algebra.
\begin{enumerate}[(1)]
\item (Proposition \ref{2tilt}) The partially ordered set of basic two-term tilting complexes of $\Lambda_\G$ is isomorphic to that of $\Lambda_{\G'}$ if $\G$ and $\G'$ have the same underlying ribbon graph.

\item (Theorem \ref{TD of type odd}) $\Lambda_\G$ is tilting-discrete if, and only if, $\G$ contains at most one cycle of odd length, and no cycle of even length.

\item (Corollary \ref{cor-type-odd-equiv}) In the case of (2), any algebra derived equivalent to  $\Lambda_\G$ is also a Brauer graph algebra $\Lambda_{\G'}$.  Moreover, $\G'$ is flip equivalent to $\G$ in the sense of \cite{A2}.
\end{enumerate}
\end{theorem}
Note that the derived equivalence class in (3) is not entirely new; for slightly more details, see the discussions at the end of Section \ref{sec-tilt-discrete}.

The classification and description of two-term tilting complexes (and their indecomposable summands) of the Brauer star algebras, i.e. (representation-finite) symmetric Nakayama algebras, have already been studied in \cite{SZI}, and implicitly in \cite{Ada}.  In the case of the Brauer tree algebras, these tasks were carried out in \cite{AZ1,Zvo1}.  We note that the result in \cite{SZI} actually classifies more than just two-term tilting complexes.  In all the mentioned articles, as well as ours, the indecomposable summands of two-term tilting complexes are classified first. Then one gives the conditions on how they can be added together to form tilting complexes.  In contrast to \cite{SZI,Zvo1}, we will not calculate the corresponding endomorphism algebra in this article, but leave it for our sequel.  For Brauer tree algebras, the classification in \cite{Zvo1} is described in terms of the Ext-quiver of the algebra, whereas ours are described by combinatorics on the defining ribbon graph.

The combinatorial language we use is heavily influenced by the ribbon graph theory.  The use of this language for Brauer graph algebras is slightly different from the traditional approach used in, for example, \cite{Ro,Kau}, but it is also not new - we learn it from the paper of Marsh and Schroll \cite{MS}, which relates Brauer graphs with surface triangulation (and $n$-angulation) and cluster theory.  The key advantage of adopting this approach is that we can clear out many ambiguities when there is a loop in the Brauer graph.  Moreover, while writing up this article, this new language gives us a glimpse of a connection between geometric intersection theory and the tilting theory of Brauer graph algebras (see Remark \ref{rmk-geometry}).  We hope to address this issue in a subsequent paper.

Our approach to the homological calculations draws results and inspirations from \cite{AI,A,AIR}.  The key technique which makes the classification of two-term (pre)tilting complexes possible is Lemma \ref{lem-non-zero-hom}.  This is a particular case of a vital lemma in relating $\tau$-tilting theory with two-term silting/tilting complexes in \cite{AIR}.  This homological property was also used in \cite{AZ1,Zvo1}.  While we will not mention and introduce $\tau$-tilting theory formally, a partial motivation of this work is in fact to obtain a large database of calculations for $\tau$-tilting theory, and hopefully to inspire further development in said theory.  The idea of investigating tilting-discreteness using only knowledge about two-term tilting complexes in this article (and the investigation of similar vein in \cite{AM}) is also inspired by the elegance of $\tau$-tilting theory.

This article is structured as follows.  We will first go through in Section \ref{sec-ribbon-combinatorics} the essential ribbon (and Brauer) graph combinatorics needed for this paper.  No mathematical prerequisites are needed to understand this section, although knowledge of basic notions in ordinary graph theory would be helpful to find intuitions. 
Section \ref{sec-alg-hom-prelim} is devoted to recall known results and notions needed to understand the algebraic and homological side of the first main theorem.
Brief reminder on the basic tilting theory is presented separately in Subsection  \ref{sec-tilting-theory}, whereas a review on Brauer graph algebras and their modules can be found in Subsection \ref{sec-BGA}.

In Section \ref{sec-twotilt}, we will first explain some elementary observations on the two-term tilting complexes of a Brauer graph algebra.  This allows us to write down maps between the homological objects (two-term tilting complexes and its direct summands) and the combinatorial objects, which form the statement of our first main result (Theorem \ref{mainbij}).

We will spend the entire Section \ref{sec-proof} to prove Theorem \ref{mainbij}.
In Section \ref{sec-application}, we explore an application of Theorem \ref{mainbij}.
We first present some preliminary material on tilting-connectedness and related notions in Subsection \ref{sec-application-prelim}.
Then we determine a sufficient condition for a Brauer graph algebra to be tilting-connected in Subsection \ref{sec-tilt-discrete}.
For the ease of readers, we also include the list of notations used in this article and the locations of their first appearances in Appendix \ref{sec-notations}.

\section*{Acknowledgment}
We owe our deepest gratitude to Osamu Iyama for many fruitful discussions, as well as the financial support for the third author's visit to Nagoya University, which nurtured this research.  This article is typed up during AC's subsequent visits at Nagoya University, and finished during the first author's visit at Uppsala University.  We are thankful for the hospitality of these institutions.
We thank the referee for pointing out missing references and known results in the literature.
We thank also Ryoichi Kase and Alexandra Zvonareva for various discussions.

\section*{Convention}
The following assumptions and conventions will be imposed throughout the article.
\begin{enumerate}[(1)]
\item For a sequence $w=(e_1,e_2,\ldots,e_n)$, a subsequence $w'$ is said to be \emph{continuous}, and denoted by $w'\subset w$\label{sym-subseq}, if it is of the form $(e_i,e_{i+1},\ldots, e_{i+k})$. 
We adopt this convention to avoid any possible confusion with terminologies used in string combinatorics (cf. Subsection \ref{subsec-string-com}).
\item The composition of maps $f:X\to Y$ and $g:Y\to Z$ is $gf:X\to Z$.
\item All algebras are assumed to be basic, indecomposable, and finite dimensional over an algebraically closed field $\Bbbk$.  We will often use $\Lambda$ to denote such an algebra.
\item We always work with finitely generated right modules, and use $\mod\Lambda$ to denote the category of finitely generated right $\Lambda$-modules.
\item We denote by $\proj\Lambda$ the full subcategory of $\mod\Lambda$ consisting of 
all finitely generated projective $\Lambda$-modules.
The bounded homotopy category of $\proj\Lambda$ is denoted by $\Kb(\proj\Lambda)$.
\item We sometimes write $\Lambda=\Bbbk Q/I$, where $Q$ is a finite quiver and $I$ is the ideal of relations.
\item A simple (resp. indecomposable projective) $\Lambda$-module corresponding to a vertex $i$ of $Q$ is denoted by $S_i$\label{sym-simple} (resp. by $P_i$\label{sym-proj}); an arrow of $Q$ is identified with a map between indecomposable projective $\Lambda$-modules.
\item For an object $X$ (in $\mod \Lambda$ or $\Kb(\proj\Lambda)$), we denote by $|X|$\label{sym-cardinal} the number of isomorphism classes of the indecomposable direct summands of $X$.
\end{enumerate}

\section{Ribbon graphs combinatorics}\label{sec-ribbon-combinatorics}

\begin{definition}
A \emph{ribbon graph}\label{sym-G} is a datum $\G=(V,H,s,\ol{\,\cdot\,},\sigma)$, where 
\begin{enumerate}
\item $V$\label{sym-V} is a finite set, where elements are called \emph{vertices}. 
\item $H$\label{sym-H} is a finite set, where elements are called \emph{half-edges}.
\item $s:H\to V$\label{sym-s} specifies the vertex $s(e)$ for which a half-edge $e$ is  emanating.
\item $\ol{\,\cdot\,}:H\to H$\label{sym-ol} is a fixed-point free involution, i.e. $\ol{e}\neq e$ and $\ol{(\,\ol{e}\,)}=e$ for all $e\in H$. 
\item $\sigma:H\to H$\label{sym-sigma} is a permutation on $H$ so that the set of $\langle\sigma\rangle$-orbits is $V$ and $s:H\to V=H/\langle\sigma\rangle$ is the induced projection.
For $v\in V$, the \emph{cyclic ordering around $v$} is the restriction of $\sigma$ to $s^{-1}(v)$, which will be denoted by $(e,\sigma(e),\ldots,\sigma^{k-1}(e))_v$\label{sym-cyclic-order-v} for some $e\in H$ with $s(e)=v$, where $k=|s^{-1}(v)|$.
Viewing this notation as a sequence up to cyclic permutation, its subsequences will be called \emph{cyclic suborderings around $v$}.
\end{enumerate}
\end{definition}

The geometric realisation of $\G$ is $(H\times[0,1])/\sim$, where $\sim$ is the equivalence relation defined by 
$(e,\ell)\sim(\ol{e},1-\ell)$ and $(e,0)\sim (\sigma(e),0)$.  
Intuitively, we ``glue" $e$ and $\ol{e}$ together to form a graph-theoretic edge (geometric line segment) incident to 
$s(e)$ and $s(\ol{e})$.
Because of this, we call an unordered pair $\{e,\ol{e}\}$ an \emph{edge} of $\G$, and the vertices (or vertex) $s(e)$ and $s(\ol{e})$ the \emph{endpoint(s)} of the edge $E$.
We also define the \emph{valency} of a vertex $v$ as $\val(v):=|s^{-1}(v)|$\label{sym-val}.
Note that a ribbon graph (or its geometric realisation) is also called a locally embedded graph or a fatgraph in the literature.

A \emph{Brauer graph} is a ribbon graph equipped with a \emph{multiplicity function} $\mm:V\to \mathbb{Z}_{>0}$\label{sym-mm} 
which assigns a positive integer, called \emph{multiplicity}, to each vertex of $\G$.

We will impose the following assumptions and conventions on ribbon graphs throughout.
\begin{itemize}
\item A Brauer/ribbon graph is always assumed to be connected, i.e. its geometric realisation is connected.

\item We will always denote half-edges by small alphabets such as $e\in H$, 
and we will denote the corresponding edge $\{e,\ol{e}\}$ by the same letter in capital, which is $E$ here.

\item According to the definition of Brauer graph algebras (see Definition \ref{def-BGA}), 
we usually present $\G$ graphically using its geometric realisation with the cyclic ordering of (half-)edges around 
each vertex presented in the counter-clockwise direction.

\item Having said that, whenever we present a local structure such as
\begin{align}
\xymatrix{
&  {v} \ar@{-}[r]^{e_2}\ar@{-}[l]_{e_1}&
}\notag
\end{align}
then the two lines ($e_1$ and $e_2$) emanating from $v$ are regarded as half-edges emanating from $v$.
\end{itemize}

\begin{example}\label{eg-two-loop}
Let $\G$ be the ribbon graph:
\[
(\{v\},\;\; \{e,\ol{e},f,\ol{f}\},\;\; s\equiv v,\;\; \ol{\,\cdot\,},\;\;(e,\ol{e},\ol{f},f)_v),
\]
Then the geometric realisation of $\G$ is:
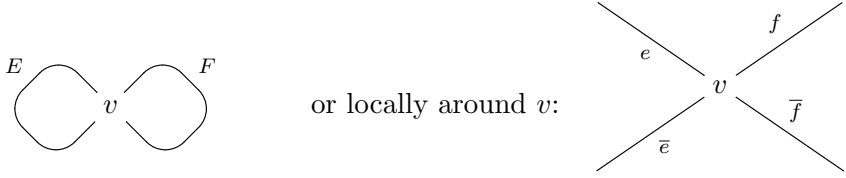
\begin{align}
\xymatrix@R=30pt@C=30pt{ & & & \\
&  {v} \ar@{-} `ul[l] `dl[d]_{E}  `dr[] []  \ar@{-} `ur[r] `dr[d]^{F}  `dl[] []
& &\text{or locally around $v$: }  \\
& & &} \xymatrix@R=25pt@C=40pt{ & &\\
 &  {v} \ar@{-}[lu]^{e} \ar@{-}[ld]^{\ol{e}}  \ar@{-}[rd]^{\ol{f}} \ar@{-}[ru]^{f} & \\
& & } \notag 
\end{align}
\end{example}

\begin{definition}[Half-walk, walk, and signed walk]\label{def-walks}
\begin{itemize}
\item[(a)] A non-empty sequence $w$\label{sym-halfwalk} of half-edges $(e_1,\ldots,e_l)$ such that 
$s(e_{i+1})=s(\ol{e_i})$ for all $i=1,\ldots,l-1$ is called a \emph{half-walk}.
Defining $\ol{w}=(\ol{e_l},\ldots,\ol{e_1})$ makes $\ol{\,\cdot\,}$ an involution on the set of half-walks.

\item[(b)] A \emph{walk} on a ribbon graph $\G$ is the unordered pair $W=\{w,\ol{w}\}$ of half-walks.  For exposition convenience, by ``a walk $W$ given by (the half-walk) $w$" we mean that $W=\{w,\ol{w}\}$\label{sym-walk}.

\item[(c)] For a half-walk $w=(e_1,\ldots,e_l)$, we define $s(w):=s(e_1)$ and $s(\ol{w})=s(\ol{e_l})$ 
as the \emph{endpoint(s)} of $W=\{w,\ol{w}\}$.  We say that a vertex $v$ (resp. half-edge $e$) is in $W$ if $v=s(\ol{e_l})$ or $v=s(e_i)$ (resp. $e=e_i$ or $\ol{e_{i}}$) for some $i\in\{1,\ldots,l\}$.

\item[(d)] A \emph{signature} on a walk $W=\{w=(e_1,\ldots,e_l),\ol{w}\}$ is an assignment of signs $\e_W(e)=\e_W(\ol{e})$\label{sym-signature} on the half-edges in $W$ such that $\e_W(e_i)\neq \e_{W}(e_{i+1})$ for all $i\in\{1,\ldots,l-1\}$. 
A walk equipped with a signature is called a \emph{signed walk}. A \emph{signed half-walk} is a half-walk $w$ equipped with a signature on $W=\{w,\ol{w}\}$.  We often denote a signed half-walk by $(w;\e)$\label{sym-signed-hw} or $(e_1^{\e(e_1)},e_2^{\e(e_2)},\ldots,e_l^{\e(e_l)})$.
Denote by $\SW(\G)$\label{sym-SWG} the set of signed walks of $\G$.
\end{itemize}
\end{definition}
Note that the notation of half-walk has no subscript next to the closing bracket, whereas the subscript $v$ in $(e_1,e_2,\ldots,e_k)_v$ clarifies that this sequence represents a cyclic subordering around $v$.

Note that our ``walks" here can be visualised as walks in graph theory, if we regard (the geometric realisation of) $\G$ as a graph of undirected edges. 
However, they are not exactly the same as graph-theoretic walks, as we can see in the following example.

\begin{example}
Let $\G$ be the ribbon graph in the previous example.
There are four signed walks induced by a graph-theoretic walk (with signs) $(v,E^+,v,F^-,v)$. 
They are given by the half-walks $w_1=(e^+,f^-)$, $w_2=(e^+,\ol{f}^-)$, $w_3=(\ol{e}^+,f^-)$, $w_4=(\ol{e}^+,\ol{f}^-)$.

Note that one can not always find a signature for a walk.  For example, the walk $(e,\ol{e})$ in this example cannot have a signature.  On the other hand, if one can define a signature on a walk, then there are exactly two choices of signature.
\end{example}

The combinatorial gadget is named in the spirit of the Green's walk around a Brauer tree \cite{Gre}, which is essentially a combinatorial description of the minimal projective resolution of a ``maximal non-projective uniserial module" in the sense of \cite[Def 2.1]{Ro}; see also Example \ref{eg-hook}.
We will see in Section \ref{sec-twotilt} that the signed walks coincide with the minimal projective \emph{presentations} of certain modules of a Brauer graph algebra.

\begin{definition}[Common subwalk] \label{def-common-subwalk}
\begin{enumerate}
\item We denote by $w\cap w'$\label{sym-common-subwalk} the set of half-walks $z$ such that $z\subset w$, $z\subset w'$, and
there is no $z'\neq z$ with $z\subset z'\subset w$ and $z\subset z'\subset w'$.

\item A \emph{subwalk} $Z$ of $W$ is a walk given by some continuous subsequence $z\subset w\in W$.

\item A \emph{common subwalk} $Z$ of $W$ and $W'$ is a walk given by $z$ with $z\subset w$ and $z\subset w'$ for some $w\in W$ and $w'\in W$.

\item We denote by $W\cap W'$\label{sym-set-common-subwalk} the set of common subwalks $Z$ given by $z\in w\cap w'$, for some $w\in W$ and $w'\in W'$.  An element of $W\cap W'$ is called a \emph{maximal common subwalk} of $W$ and $W'$.
\end{enumerate}
\end{definition}

To ease our burden of explaining the combinatorics for various definitions and proofs,
we will attach some extra data to a signed half-walk $w=(e_1,\ldots,e_l;\e_W)\in W$ 
which are uniquely determined by the signature of $W$.

\begin{definition}
A \emph{virtual (half-)edge} is an element in the set $\{\vir_-(e),\vir_+(e)\mid e\in H\}$\label{sym-vir}.  We can augment $s$ on the virtual edges so that $s(\vir_\pm(e))=s(e)$.
Let $v$ be $s(e)$ and suppose $(e_1,\ldots,e_k)_v$ is the cyclic ordering around $v$.  
We define the \emph{cyclic ordering (around $v$) accounting the virtual edges} as
\begin{align}\label{eq-cyc-ord-vir}
(\vir_-(e_1),e_1,\vir_+(e_1),\vir_-(e_2),e_2,\vir_+(e_2),\ldots,\vir_-(e_k),e_k,\vir_+(e_k))_v.
\end{align}
A (possibly non-continuous) subsequence of (\ref{eq-cyc-ord-vir}) is called \emph{cyclic subordering around $v$ accounting the virtual edges}.
Suppose  a signed walk $W$ is given by $w=(e_1,\ldots,e_l; \e_W)$.  We define the following virtual edges attached to $W$:
\begin{align}
e_0 = \ol{e_0} & := \vir_{-\e_W(e_1)}(e_1) \notag \\
e_{l+1} = \ol{e_{l+1}} & := \vir_{-\e_W(e_l)}(\ol{e_l})\notag 
\end{align}
We also define $\e_W(e_0)=\e_W(\ol{e_0})=-\e_W(e_1)$ and $\e_W(e_{l+1})=\e_W(\ol{e_{l+1}})=-\e_W(e_l)$.
\end{definition}

From now on, one should always bear in mind the following convention:
\emph{
\begin{center}
A cyclic subordering around an endpoint of a (half-)walk involving $e_0$ and/or $e_{l+1}$ is always regarded as the cyclic subordering accounting the virtual edges.
\end{center}
}
Unless otherwise specified, we fix $W$ and $W'$ as two (not necessarily distinct) signed walks given by half-walks
$w=(e_1,\ldots,e_m;\e_W)$ and $w'=(e_1',\ldots,e_n';\e_{W'})$ respectively.
Moreover, it is automatically understood what we mean by $e_0,e_{m+1},e'_0,e'_{n+1}$, etc. from the definition of virtual edges.

\begin{definition}[Sign condition]
We say that a signed walk $W$ (or a signed half-walk $w\in W$) \emph{satisfies the sign condition} if $\e_W(e_1)=\e_W(e_m)$ whenever $s(e_{1})=s(\ol{e_{m}})$.
In general, we say that two walks $W, W'$ satisfy the sign condition if all of the following conditions are satisfied:
\begin{itemize}
\item $\e_W(e_1)=\e_{W'}(e'_1)$ if $s(e_1)=s(e'_1)$;
\item $\e_W(e_1)=\e_{W'}(e'_n)$ if $s(e_1)=s(\ol{e'_n})$;
\item $\e_W(e_m)=\e_{W'}(e'_1)$ if $s(\ol{e_m})=s(e'_1)$;
\item $\e_W(e_m)=\e_{W'}(e'_n)$ if $s(\ol{e_m})=s(\ol{e'_n})$.
\end{itemize}
\end{definition}

\begin{definition}[Non-crossing condition at a maximal common subwalk]
Let $Z$ be a maximal common subwalk of $W$ and $W'$ given by $z=(t_1,t_2,\ldots, t_\ell)\in w\cap w'$ so that $t_{k}=e_{i+k-1}=e'_{j+k-1}$ for all $k\in\{1,\ldots,\ell\}$.  
Then we call the pair of cyclic suborderings on $\{\ol{e_{i-1}},\ol{e_{j-1}'},t_1\}$ and $\{e_{i+\ell},e_{j+\ell},\ol{t_\ell}\}$ the \emph{neighbourhood cyclic orderings} of $Z$.
We say that $W$ and $W'$ are \emph{non-crossing at $Z$} if the following holds:
\begin{itemize}
\item[(NC1)] $\e_W(t_{k})=\e_{W'}(t_{k})$ for each $k\in\{ 1,\ldots, \ell\}$.
\item[(NC2)] With the exception of $i=j=1$ and/or $m+1-i-\ell=n+1-j-\ell=0$ (i.e. $u$ and/or $v$ being the endpoint(s) of \emph{both} walks), the neighbourhood cyclic orderings of $Z$ are
\begin{align}
\textnormal{either } &(t_1,\ol{e_{i-1}},\ol{e'_{j-1}})_{s(z)}\ \textnormal{ and }\ (\ol{t_\ell},e'_{j+\ell},e_{i+\ell})_{s(\ol{z})}\ \textnormal{ respectively}, \notag\\
\textnormal{or }\ &(t_1,\ol{e'_{j-1}},\ol{e_{i-1}})_{s(z)}\ \textnormal{ and }\ (\ol{t_\ell},e_{i+\ell},e'_{j+\ell})_{s(\ol{z})} \ \textnormal{ respectively}. \notag
\end{align}

Visualising the two cases locally:
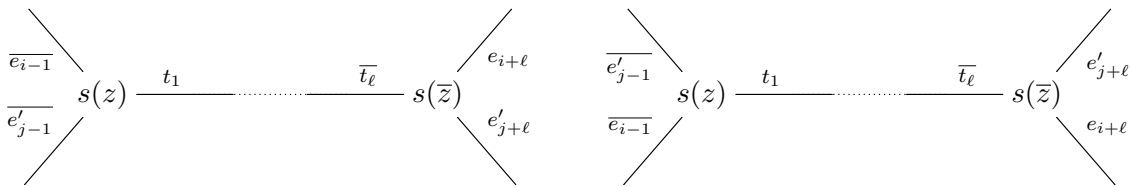
\begin{align}
\xymatrix@C=15pt{
 \ar@{-}[rd]_{\ol{e_{i-1}}} & & & & & & & \ar@{-}[ld]^{e_{i+\ell}} \\
  &  {s(z)} &  \ar@{.}[rrr]& \ar@{-}[ll]_{t_{1}} & \ar@{-}[rr]^{\ol{t_\ell}} & &  {s(\ol{z})}\\
\ar@{-}[ru]^{\ol{e'_{j-1}}} & & & & & & &\ar@{-}[lu]_{e'_{j+\ell}}
}\qquad \xymatrix@C=15pt{
 \ar@{-}[rd]_{\ol{e'_{j-1}}} & & & & & & & \ar@{-}[ld]^{e'_{j+\ell}} \\
  &  {s(z)} & \ar@{.}[rrr]&  \ar@{-}[ll]_{t_{1}} & \ar@{-}[rr]^{\ol{t_\ell}} & & {s(\ol{z})}\\
\ar@{-}[ru]^{\ol{e_{i-1}}} & & & & & & &\ar@{-}[lu]_{e_{i+\ell}}
}\notag
\end{align}
\end{itemize}
\end{definition}
\medskip

For a vertex in $v=s(e_i)$ in $W$, we refer to the set $\{\ol{e_{i-1}},e_i\}$ as a \emph{neighbourhood} of $v$ in $W$.
Suppose $\{a,b\}$ and $\{c,d\}$ are neighbourhoods of $v$ in $W$ and in $W'$ respectively.
We say that $v$ is an \emph{intersecting vertex} of $W$ and $W'$, if $a,b,c,d$ are pairwise distinct.
Note that $v$ can be an intersecting vertex with respect to multiple different pairs of neighbourhoods.

\begin{definition}[Non-crossing condition at an intersecting vertex and admissible walks]\label{def-NC-admissible}
We say that $W$ and $W'$ are \emph{non-crossing at the intersecting vertex $v$} if the following condition is satisfied.
\begin{itemize}
\item[(NC3)] If $v$ is an intersecting vertex with respect to the neighbourhoods $\{a,b\}$ in $W$ and $\{c,d\}$ and $W'$, and at most one of $a,b,c,d$ is virtual, then we have either one of the following cyclic subordering around $v$ and signature:
\begin{align}
\notag (a^+,b^-,c^+,d^-)_v, \quad \;\; (a^+,b^-,c^-,d^+)_v.
\end{align}
The local structure around $v$ for these two conditions can be visualised as
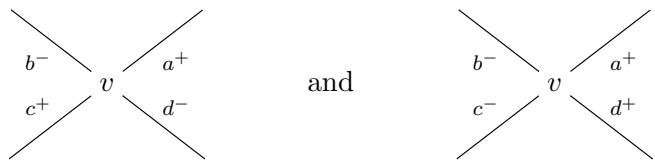
\begin{align}
\xymatrix@R=20pt@C=30pt{
\ar@{-}[dr]_{b^-} & & \ar@{-}[dl]^{a^+} & & \ar@{-}[dr]_{b^-} & & \ar@{-}[dl]^{a^+}\\
 &  {v} &  & \text{and} &  &  {v} &  \\
\ar@{-}[ur]^{c^+} & & \ar@{-}[ul]_{d^-} & & \ar@{-}[ur]^{c^-} & & \ar@{-}[ul]_{d^+}
}\notag
\end{align}
respectively.
\end{itemize}

We say that two walks $W,W'$ are \emph{non-crossing}, or they satisfy the \emph{non-crossing conditions}, if they are non-crossing at all maximal common subwalks and all intersecting vertices.

If in addition $W=W'$, we may specify that $W$ is \emph{self-non-crossing}.  An \emph{admissible walk} is a self-non-crossing (signed) walk which also satisfies the sign condition.  Denote by $\AW(\G)$\label{sym-AWG} the subset of $\SW(\G)$ consisting of all admissible walks.
\end{definition}

\begin{remark}\label{rmk-geometry}
While it is easy to see why (NC2) is called non-crossing, it is not apparent that why (NC1) is a non-crossing condition but not a sign condition, and why such signatures on half-edges around $v$ are required in (NC3) to make them non-crossing.  Although we will not use this representation of signed walks in this article, the correct graphical (geometrical) realisation is as follows.  When we go along the signed (half-)walk, say $(e_i,e_{i+1})\subset w$ with $i\neq 0$ and $i\neq m+1$, instead of visualising the situation as a line passing through a vertex:
\begin{align}
\xymatrix@C=40pt{ &  {v} \ar@{-}[l]_{\ol{e_i}} \ar@{-}[r]^{e_{i+1}}&}
\notag
\end{align}
we think of the vertex $v=s(e_{i+1})$ as lying below (resp. above) the line (relative to this presentation) if $\e_W(e_{i+1})=-$ (resp. $\e_W(e_{i+1})=+$).  

This ``correct" visualisation is in fact a generalisation of the technique used in \cite{KS,ST} - from a disc to compact orientable Riemann surfaces with marked point and boundaries (specified by $\G$).  In this geometric setting, our signed walks become certain type of curves, and the non-crossing condition translates into requiring a curve to have no self-intersection.  We hope to explain this geometric setting with more details in a sequel paper.
\end{remark}

Checking the non-crossing condition is in practice much easier than the way it is defined here - draw the walk around the geometric realisation of $\G$ in a non-crossing way and put signs on the edges to check (NC3).  We briefly explain how to check non-crossing-ness algorithmically here, and recommend the reader to carefully go through the proof of the next proposition in order to familiarise with the procedure.

First, fix some $w\in W$ and $w'\in W'$.  Start with a vertex, say $v=s(e_i)=s(e'_j)$, which is in both $w$ and $w'$.  In the special case of $W=W'$, one can simply start with a vertex which appears at least twice in (any) $w\in W$.  If the half-edges $\ol{e_{i-1}},e_i,\ol{e_{j-1}'},e_j'$ are pairwise distinct, then $v$ is an intersecting vertex with respect to the neighbourhoods given by these half-edges.
One can verify if the signed cyclic ordering required by (NC3) is satisfied simultaneously.
Otherwise, we have a common subwalk given by the coinciding half-edge, and one can expand this half-edge to some $z\in w\cap w'$ or $z\in w\cap \ol{w'}$.  Iterate this procedure for all possible pair of $i,j$ with $s(e_i)=s(e'_j)$.

We give an example on verifying the sign condition and the non-crossing condition:
\begin{example}\label{eg-walk1}
Let $\G$ be a ribbon graph whose geometric realisation is given on the left-hand side of the following picture.  For readability, we label the half-edges with numeral instead of letters and place the labelling next to their respective emanating vertex.
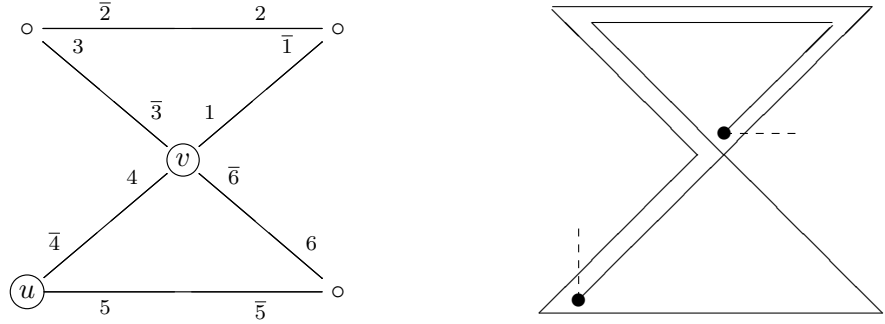
\begin{center}
\begin{picture}(300,120)
\put(0,105){$\xymatrix@C=20pt@R=16pt{
\circ\ar@{-}[rr]^{\ol{2}}\ar@{-}[rrrr]\ar@{-}[rd]^{3}\ar@{-}[rdrd] &&\ar@{-}[rr]^{2}&&\circ\ar@{-}[ldld]\\
&\ar@{-}[rd]^{\ol{3}}&&\ar@{-}[ru]^{\ol{1}}&\\
&&*+[Fo]{v}\ar@{-}[rd]^{\ol{6}}\ar@{-}[rdrd]\ar@{-}[ru]^{1}&&\\
&\ar@{-}[ru]^{4}&&\ar@{-}[rd]^{6}&\\
*+[Fo]{u}\ar@{-}[rr]_{5}\ar@{-}[ruru]\ar@{-}[ru]^{\ol{4}}\ar@{-}[rrrr]&&\ar@{-}[rr]_{\ol{5}}&&\circ
}$}
\put(270,68){\circle*{5}}
\put(270,68){\line(1,1){41}}
\put(220,110){\line(1,0){91}}
\put(220,110){\line(1,-1){110}}
\put(200,0){\line(1,1){60}}
\put(200,0){\line(1,0){130}}
\put(260,60){\line(-1,1){55}}
\put(205,116){\line(1,0){121}}
\put(215,5){\line(1,1){111}}
\put(215,5){\circle*{5}}
\multiput(270,68)(6,0){5}{\line(1,0){3}}
\multiput(215,5)(0,6){5}{\line(0,1){3}}
\end{picture}
\end{center}

Define $w_1$ as the following signed half-walk.
\begin{align}
w_{1}=(1^{+},2^{-},3^{+},\ol{6^{-}},\ol{5^{+}},\ol{4^{-}},\ol{3^{+}},\ol{2^{-}},\ol{1^{+}},4^{-}). \notag
\end{align}
Applying the involution to them we get:
\begin{align}
\ol{w_{1}}=(\ol{4^{-}},1^{+},2^{-},3^{+},4^{-},5^{+},6^{-},\ol{3^{+}},\ol{2^{-}},\ol{1^{+}}).\notag
\end{align}

The walk $W_1$ is shown on the right-hand side of the above picture. We represent the virtual edges as dashed lines to indicate their relative position with respect to the (augmented) cyclic ordering around $u$ and $v$.  From this picture, one then expects that $W_1$ is self-crossing.  We show how this can be checked properly.

We have the following sets:
\begin{align}
w_1\cap w_1 &= \{w_1\},\quad w_1\cap \ol{w_1}= \{z_1=(1,2,3), \ol{z_1}, z_2=(4), \ol{z_2}\},\notag
\end{align}
which gives us the following set of maximal common subwalks
\begin{align}
W_{1}\cap W_{1}&=\left\{\; W_{1},\; Z_1=\{ z_1, \ol{z_1} \},\; Z_2=\{ z_2,\ol{z_2} \}\; \right\}\notag.
\end{align}
It is easy to see that (NC1) is satisfied in all the cases.  In the above picture, $Z_1$ is represented by the overlapping upper triangle, whereas $Z_2$ is represented by the overlapping edge in the lower triangle.  Focusing on the adjacent half-edges around these overlapping parts, one can see that (NC2) does not hold at both $Z_1$ and $Z_2$.

We will be more precise here.  The neighbourhood cyclic orderings around $s(z_1)$ and $s(\ol{z_1})$ in the half-walk $z_1$ are 
\begin{align}
(1,4,\vir_{-}(1))_{v=s(z_1)} \quad \text{and} \quad (\ol{3},4,\ol{6})_{v=s(\ol{z_1})} \quad \text{respectively.}\notag 
\end{align}
For $z_2$, we have instead
\begin{align}
(4,1,\ol{3})_{v=s(z_2)}\quad \text{and}\quad (\ol{4},\vir_{+}(\ol{4}),5)_{u=s(\ol{z_2})}. \notag
\end{align}
Therefore, (NC2) does not hold in both cases - as we have claimed.

From the picture above, it is easy to see that $v$ is the only self-intersecting vertex, and we leave it for the reader to confirm the claim.  There are only four neighbourhoods of $v$ in $W_{1}$:
\begin{align}
\{ \vir_{-}(1), 1\}, \quad \{ 1,4\},\quad \{ \ol{3}, 4\}, \quad\text{and }\{ \ol{3}, \ol{6}\}. \notag
\end{align}

Then, to check (NC3), we only need to observe the three cases 
$\{ \vir_{-}(1), 1, \ol{3},\ol{6}\}$, $\{ \vir_{-}(1), 1, \ol{3},4\}$, and $\{ 1, 4, \ol{3},\ol{6}\}$.
The associating signed cyclic suborderings are given by 
\begin{align}
(\ol{3^{+}},\ol{6^{-}},\; \vir_{-}(1)^{-},1^{+})_{v},\;
(\ol{3^{+}},4^{-},\vir_{-}(1)^{-}, 1^{+},)_{v},\;\text{and }
(1^{+},\ol{3^{+}},4^{-},\ol{6^{-}})_{v} \text{ respectively}. \notag
\end{align}
Thus, while the first two cases comply with (NC3), the third one does not.
\end{example}

The key result of this section is the following one.

\begin{proposition}\label{type-odd}
Let $\G$ be a ribbon graph. Then the following are equivalent.
\begin{itemize}
\item[(1)] $\AW(\G)$ is a finite set.
\item[(2)] $\G$ consists of at most one odd cycle and no even cycle.
\end{itemize}
\end{proposition}

\begin{remark}
A cycle of length $n$ in $\G$ is an $n$-gon embedded in $\G$, i.e. there is a walk $(e_1,\ldots,e_k)$ in $\G$ with endpoints being the same, and no repeating vertices along the walk.  An odd cycle (resp. even cycle) is a cycle of odd (resp. even) length.  Note that a 1-gon (cycle of length 1) is a loop.  We denote a cycle in $\G$ by a sequence $\mathcal{C}=(v,E_1,v_1,E_2,\ldots, v_{k-1},E_k,v)$, so that $s(e_1)=v=s(\ol{e_k})$, $s(e_i)=v_{i-1}=s(\ol{e_{i+1}})$.
\end{remark}

For ease of exposition, regardless of whether the entries of a sequence are half-walks or half-edges, or a mixture of both, the sequence is understood as the half-walk given by concatenating all the data in the obvious way.

\begin{proof}[Proof of (2)$\Rightarrow$(1)]
Suppose $\G$ is a graph containing at most one odd cycle and no even cycle.
Since sign alternates as we go along a signed walk, the same edge never appears more than once.  In particular, the 
length $\ell$ of the sequence defining a signed walk is less than or equal to the number of edges in $\G$.  Moreover, for a given $\ell$, since there are only finitely many sequences (of half-edges) of length $\ell$, $\SW(\G)$ is finite.  Hence, the subset $\AW(\G)$ of $\SW(\G)$ is also finite.
\end{proof}

We will spend the rest of the section to show (1)$\Rightarrow$(2).
For better readability, we put all the labellings of half-walks and walks in bold face.

First note that if $\G$ does not satisfy the condition of (2), then we can assume it contains an even cycle $\mathcal{C}=(v,E_1,v_1,E_2,\ldots, v_{k-1},E_k,v)$, or it contains two odd cycles $\mathcal{C}=(u,E_{1},u_{1},E_{2},\ldots,u_{m-1},E_{m},u)$ and $\mathcal{C}'=(v,F_{1},v_{1},F_{2},\ldots,v_{n-1},F_{n},v)$.
We can choose the cycles with minimal length.
In the former case, this means that we can assume that none of the edges is a loop; otherwise, we get a smaller cycle or two odd cycles.
In the latter case, the two cycles are connected by a line $\mathcal{L}$ with endpoints $u,v$.
We can also assume that $\mathcal{L}$ is chosen with the minimal number of (non-loop) edges, and that it intersects $\mathcal{C}$ (resp $\mathcal{C}'$) at a single vertex $u$ (resp. $v$).
Note that we allow $\mathcal{L}$ to be just a single point (i.e. $u=v$).
The minimal length assumption implies that the vertices that are in $\mathcal{C}$ (resp. $\mathcal{C}'$) but not $\mathcal{L}$ does not appear twice in the cycle; otherwise, we get a cycle of shorter length. 
Let us be more precise now about these configurations.

\begin{enumerate}[(I)]
\item $\G$ has an even cycle $\mathcal{C}=(v,E_1,v_1,E_2,\ldots, v_{k-1},E_k,v)$ where none of the $E_i$'s is a loop.
$\mathcal{C}$ has the following geometric realisation in $\G$:
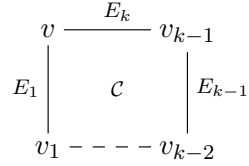
\begin{align}\notag
\xymatrix@R=12pt@C=12pt{\ar@{}[rrdd]|{\mathcal{C}}
v\ar@{-}[dd]_{E_{1}}&&v_{k-1}\ar@{-}[ll]_{E_{k}}\\
&&\\
v_{1}\ar@{--}[rr]&&v_{k-2}\ar@{-}[uu]_{E_{k-1}}
}\end{align}
We label the half-edges by $e_i$ so that $s(e_i)=v_{i-1}$.  

\item $\G$ has two odd cycles $\mathcal{C}=(u,E_{1},u_{1},E_{2},\ldots,u_{m-1},E_{m},u)$ and $\mathcal{C}'=(v,F_{1},v_{1},F_{2},\ldots,v_{n-1},F_{n},v)$; they are connected by a line $\mathcal{L}=(D_1, \ldots, D_l)$ of length $l\geq 0$ and with endpoints $u,v$.
$\mathcal{L}$ is assumed to have minimal length amongst all lines connecting $\mathcal{C}$ and $\mathcal{C}'$.
Apart from $u$ and $v$, all the vertices appearing in the union of these subgraphs are pairwise distinct.
This subgraph has one of the following geometric realisations:
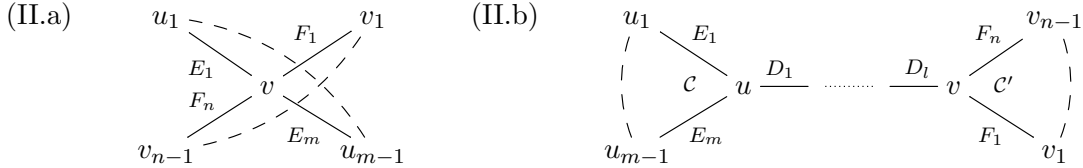
\begin{align}
\xymatrix@R=12pt@C=18pt{\text{(II.a)} & u_1\ar@{--}@/^1pc/[rrdd]& & v_{1}\ar@{--}@/^1pc/[lldd] & \text{(II.b)} & \ar@{}[rdd]|{\mathcal{C}} 
u_{1}\ar@{--}@/_/[dd]&&&&&v_{n-1}\\
&& v\ar@{-}[lu]^{E_1}\ar@{-}[rd]_{E_m}\ar@{-}[ru]^{F_1}\ar@{-}[ld]_{F_n} & & & & u\ar@{-}[lu]_{E_{1}}\ar@{-}[ld]^{E_{m}}\ar@{-}[r]^{D_{1}}&\ar@{.}[r]&\ar@{-}[r]^{D_{l}}&v\ar@{-}[ru]^{F_{n}}\ar@{-}[rd]_{F_{1}}&\\
& v_{n-1} &  & u_{m-1} & & u_{m-1}&&&&&v_{1}\ar@{--}@/_/[uu] \ar@{}[luu]|{\mathcal{C}'}
}\notag
\end{align}
We label the half-edges as follows.  
For (II.a), we require that the cyclic ordering around $v$ is given by $(\pmb{e_1},\overline{\pmb{f_n}},\ol{\pmb{e_m}},\pmb{f_1})_v$.  
For (II.b), we require cyclic orderings $(\pmb{e_1},\ol{\pmb{e_m}},\pmb{d_1})_u$ and $(\pmb{f_1},\ol{\pmb{f_n}},\ol{\pmb{d_l}})_v$ if $l>0$, or $(\pmb{e_1},\ol{\pmb{e_m}},\pmb{f_1},\ol{\pmb{f_n}})_v$ if $l=0$.   The remaining labellings are then fixed (uniquely) in the way that the following signed half-walks in $\G$ can be defined:
\begin{align}
\pmb{d}:=(\pmb{d_{1}}^{-},\pmb{d_{2}}^{+},\ldots, \pmb{d_{l-1}}^{-\e},\pmb{d_{l}}^{\e}) \notag \\
\pmb{e}:=(\pmb{e_{1}}^{+},\pmb{e_{2}}^{-},\ldots, \pmb{e_{m-1}}^{-},\pmb{e_{m}}^{+}) \notag \\
\pmb{f}:=(\pmb{f_{1}}^{-\e},\pmb{f_{2}}^{\e},\ldots, \pmb{f_{n-1}}^{\e},\pmb{f_{n}}^{-\e}),\notag
\end{align}
where $\e=+$ in (II.a), and $\e$ is the sign of $(-1)^l$ in (II.b).
Here $\pmb{e}$ and $\pmb{f}$ are half-walks given by walking along respective cycles.
\end{enumerate}

To prove (1)$\Rightarrow$(2) of Proposition \ref{type-odd}, it suffices to show that there are infinitely many admissible walks in each of the three cases.

\medskip


\textbf{Case (I)}.
Let $\pmb{s}$ be the signed half-walk $(\pmb{e_{1}}^{+},\pmb{e_{2}}^{-},\ldots, \pmb{e_{k}}^{-})$.
For a positive integer $p$, we define a signed half-walk $\pmb{w}_{p} := (v^{(0)},\pmb{s},v^{(1)},\pmb{s},\ldots,\pmb{s},v^{(p)},\pmb{e_{1}}^{+},v_{1})$ with $v^{(i)}=v$ for each $i\in\{0,1\ldots,p\}$, i.e. the concatenation of $p$ copies $\pmb{s}$ and a copy of $\pmb{e_{1}}$.
We are going to show that its corresponding signed walk $\pmb{W}_p$ is admissible.

Let us look at some concrete examples to convince ourselves that why these signed walks should be admissible.
Suppose the length of $\mathcal{C}$ is 4.
Then $\pmb{W}_1$ and $\pmb{W}_2$ can be visualised as follows.
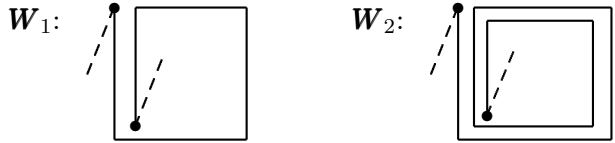
\begin{center}
\begin{picture}(200,60)(40,0)
\thicklines
\put(0,47){$\pmb{W}_1$:}
\put(40,55){\line(0,-1){50}}
\put(40,5){\line(1,0){50}}
\put(90,5){\line(0,1){50}}
\put(90,55){\line(-1,0){42}}
\put(48,55){\line(0,-1){45}}
\dashline[20]{5}(40,55)(30,30)
\dashline[20]{5}(48,10)(58,35)
\put(40,55){\circle*{4}}
\put(48,10){\circle*{4}}

\put(130,47){$\pmb{W}_2$:}
\put(170,55){\line(0,-1){50}}
\put(170,5){\line(1,0){58}}
\put(228,5){\line(0,1){50}}
\put(228,55){\line(-1,0){52}}
\put(176,55){\line(0,-1){45}}
\put(176,10){\line(1,0){45}}
\put(221,10){\line(0,1){40}}
\put(221,50){\line(-1,0){40}}
\put(181,50){\line(0,-1){36}}
\dashline[20]{5}(170,55)(160,30)
\dashline[20]{5}(181,14)(191,38)
\put(170,55){\circle*{4}}
\put(181,14){\circle*{4}}
\end{picture}
\end{center}
Clearly, this picture shows that $\pmb{W}_1,\pmb{W}_2$ are self-non-crossing.
The sign condition is clear by construction.

Let us show the admissibility more rigorously.
By construction, there is no self-intersecting vertex in $\pmb{W}_p$.
The set $\pmb{w}_{p}\cap\pmb{w}_{p}$ consists of $\pmb{w}_{p}$ and the following half-walks.
\begin{align}
\begin{array}{rrcl}
\pmb{z}_1: & (v^{(0)},\pmb{e_{1}}^{+},v_{1}) & = & (v^{(p)},\pmb{e_{1}}^{+},v_{1})\\
\pmb{z}_2: & (v^{(0)},\pmb{s},v^{(1)},\pmb{e_{1}}^{+},v_{1}) &=& (v^{(p-1)},\pmb{s},v^{(p)},\pmb{e_{1}}^{+},v_{1})\\
\pmb{z}_3: & (v^{(0)},\pmb{s},v^{(1)},\pmb{s},v^{(2)},\pmb{e_{1}}^{+},v_{1}) &=& (v^{(p-2)},\pmb{s},v^{(p-1)},\pmb{s},v^{(p)},\pmb{e_{1}}^{+},v_1)\\
\vdots & & \vdots &   \\
\pmb{z}_p: & (v^{(0)},\pmb{s},v^{(1)},\pmb{s},\ldots, \pmb{s},v^{(p-1)},\pmb{e_{1}}^{+},v_{1}) &=& (v^{(1)},\pmb{s},v^{(2)},\pmb{s},\ldots,\pmb{s}, v^{(p)},\pmb{e_{1}}^{+},v_{1})
\end{array}\notag
\end{align}
Since the set $\pmb{w}_{p}\cap\ol{\pmb{w}_{p}}$ is empty, $\pmb{W}_{p}\cap\pmb{ W}_{p}=\{\pmb{Z}_1,\ldots,\pmb{Z}_{p},\pmb{W}_{p}\}$.

Now we can see that the sign condition and the non-crossing conditions are satisfied at $\pmb{Z}_i$ for all $i\in\{1,\ldots,p\}$.  Hence, $\pmb{W}_{p}$ is admissible for all $p\geq 1$.  In particular, $\AW(\G)$ is an infinite set.
\medskip

\textbf{Case (II.a)}.
For any positive integer $p$, define the following signed half-walk:
\begin{align}
\pmb{w}_{p}:=(u^{(0)},\pmb{e},u^{(1)},\pmb{f},u^{(2)},\pmb{e},u^{(3)},\pmb{f},u^{(4)},\ldots,\pmb{e},u^{(2p-1)},\pmb{f},u^{(2p)},\pmb{e},u^{(2p+1)}),\notag
\end{align}
where $u^{(i)}=v$ for all $i\in\{0,1,\ldots,2p+1\}$.
Let us visualise $\pmb{W}_1, \pmb{W}_2$ in the case when $m=1=n$ (i.e. $\pmb{e}=e_1, \pmb{f}=f_1$).
For simplicity, further assume $\G$ has only two edges, then it can be embedded into a genus 1 surface, where we can visualise $\pmb{W}_1,\pmb{W}_2$ as follows.
The slanted lines represent the subwalk induced by $\pmb{e}$,  the horizontal lines represent the subwalk induced by $\pmb{f}$, and the dashed lines represent virtual edges (as in Example \ref{eg-walk1}).

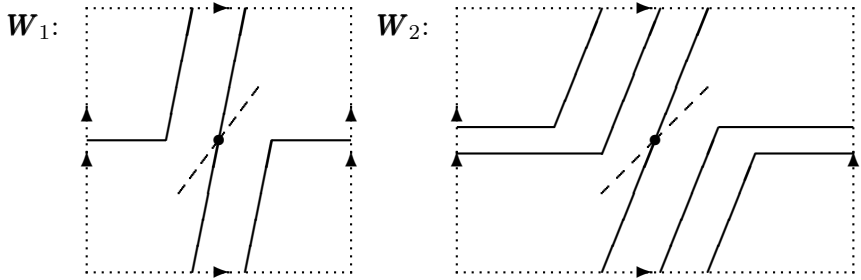
\begin{center}
\begin{picture}(320,100)(10,0)
\thicklines

\put(0,90){$\pmb{W}_1$:}
\dottedline{3}(30,0)(30,100)(130,100)(130,0)(30,0)
\put(83,0){\vector(1,0){1}}
\put(83,100){\vector(1,0){1}}
\put(30,45){\vector(0,1){1}}
\put(30,62){\vector(0,1){1}}
\put(130,45){\vector(0,1){1}}
\put(130,62){\vector(0,1){1}}
\put(80,50){\circle*{4}}
\put(80,50){\line(1,5){10}}
\dashline[20]{5}(80,50)(95,70)
\dashline[20]{5}(80,50)(65,30)
\put(90,0){\line(1,5){10}}
\put(100,50){\line(1,0){30}}
\put(30,50){\line(1,0){30}}
\put(60,50){\line(1,5){10}}
\put(70,0){\line(1,5){10}}

\put(140,90){$\pmb{W}_2$:}
\dottedline{3}(170,0)(170,100)(320,100)(320,0)(170,0)
\put(243,0){\vector(1,0){1}}
\put(243,100){\vector(1,0){1}}
\put(170,45){\vector(0,1){1}}
\put(170,62){\vector(0,1){1}}
\put(320,45){\vector(0,1){1}}
\put(2320,62){\vector(0,1){1}}
\put(245,50){\circle*{4}}
\put(245,50){\line(2,5){20}}
\dashline[20]{5}(245,50)(265,70)
\dashline[20]{5}(245,50)(225,30)
\put(265,0){\line(2,5){18}}
\put(283,45){\line(1,0){37}}
\put(170,45){\line(1,0){55}}
\put(225,45){\line(2,5){22}}
\put(247,0){\line(2,5){22}}
\put(269,55){\line(1,0){51}}
\put(170,55){\line(1,0){37}}
\put(207,55){\line(2,5){18}}
\put(225,0){\line(2,5){20}}
\end{picture}
\end{center}
Clearly, the pictures show a self-non-crossing walk (in the ordinary sense), but we need to check the condition (NC3) more carefully.
Using the cyclic ordering $(\vir_-(e_1), e_1^+, \ol{f_1}^-, \vir_-(\ol{e_1}), \ol{e_1}^+, f_1^-)_v$, we can see that (NC3) is satisfied in the only vertex $v$ of $\G$.
Sign condition is clearly satisfied by construction.
Hence, we have $\pmb{W}_1, \pmb{W}_2$ being admissible. 

Let us now show that $\pmb{W}_p$ is admissible rigourously for all $p\geq 1$ and odd $m,n\geq 1$.
Since $v$ is the only vertex appearing in both $\mathcal{C}$ and $\mathcal{C}'$, it is the only self-intersecting vertex of $\pmb{W}_{p}$.
There are only four possible different neighbourhoods of $v$ in $\pmb{W}_{p}$, given by $v=u^{(x)}$ with $x\in I:=\{0,1,2,2p+1\}$.
In particular, it makes sense to say ``the neighbourhood of $u^{(x)}$", or ``the neighbourhood at $x$".

Consider $a,b\in I$ with $a<b$.
Now one can simply draw the neighbourhoods (like in the definition of (NC3)) to see that $v$ is a self-intersecting vertex with respect to the neighbourhoods at $a$ and $b$ if, and only if, $b-a$ is odd.
One can simultaneously check that the signed cyclic ordering in (NC3) holds by considering all such $(a,b)\neq(0,2p+1)$.
The case $(a,b)=(0,2p+1)$ gives coinciding endpoints of $\pmb{W}_p$, so we need to check the sign condition, which clearly holds.

To determine maximal common subwalks, (a generalisation of) our previous observation says that we only need to consider $a,b\in\{0,1,\ldots,2p+1\}$ with $a<b$ and $b-a$ even.
Set $\pmb{s}:=(\pmb{e},\pmb{f})$ so that we can write $\pmb{w}_p$ in the form
$(v^{(0)},\pmb{s},v^{(1)},\pmb{s},\ldots,\pmb{s},v^{(p)},\pmb{e},v^{(p+1)})$, 
where $v^{(i)}=u^{(2i)}$ for all $i\in\{0,1,\ldots,p\}$, and $v^{(p+1)}=u^{(2p+1)}$.

Consider the sequences $\pmb{z}_1,\ldots,\pmb{z}_p,\pmb{w}_p$ in Case (I).
We then replace the symbol $\pmb{e_1}$ by $\pmb{e}$, and replace also the symbol $v_1$ by $v^{(p+1)}$.
Now the resulting sequences of symbols becomes walk in the setting of Case (II).
These are precisely the maximal common subwalks of $\pmb{w}_p$ with itself, obtained by expanding $v^{(a)}=v^{(b)}$; they are all of the maximal common subwalks.
One can check that the corresponding (NC1) and (NC2) conditions hold.
This implies that $\AW(\G)$ is infinite.

\medskip
\textbf{Case (II.b).}
We define $\pmb{s}$ as the signed half-walk $(\pmb{d},\pmb{f},\ol{\pmb{d}})$.
For every positive integer $p$, we define the following signed half-walk $\pmb{w}_{p}$, and show that it is admissible.
\begin{align}
\pmb{w}_{p}:=& (u^{(0)},\pmb{e},u^{(1)},\pmb{s},u^{(2)},\pmb{e},u^{(3)},\pmb{s},\ldots, \pmb{s},u^{(2p)},\pmb{e},u^{(2p+1)}, \notag \\
& \quad \ol{\pmb{s}}, u^{(2p+2)}, \ol{\pmb{e}}, u^{(2p+3)}, \ol{\pmb{s}}, u^{(2p+4)}, \ol{\pmb{e}}, u^{(2p+5)}, \ldots, \ol{\pmb{s}}, u^{(4p)}, \ol{\pmb{e}}, u^{(4p+1)}  ) , \notag
\end{align}
where $u^{(i)}=u$ for any $i\in \{ 0,1,2,\ldots, 4p+1\}$.

For guiding examples, we visualise $\pmb{w}_1$ and $\pmb{w}_2$ in the case when $m=n=3, l=1$ as follows.
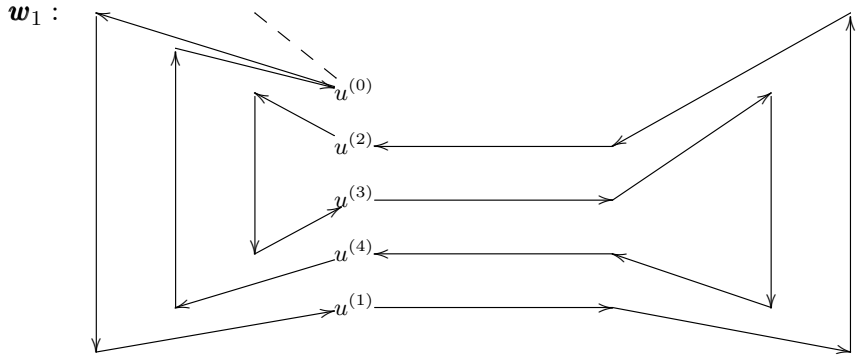
\[
\pmb{w}_1: \quad \footnotesize
\xymatrix@C=30pt@R=12pt@M=0pt{
\ar[ddddddd] & & & & & & & & &\ar[lllddd]\\
  & \ar[rrd]& & & & & & & &\\
  &   & \ar[ddd] & u^{(0)}\ar[llluu]\ar@{--}[luu] & & & & & \ar[dddd]& \\
  &   &    & u^{(2)}\ar[lu] & & & \ar[lll] &&& \\
  &   &    & u^{(3)}\ar[rrr]& & & \ar[rruu]& && \\
  &   & \ar[ru] & u^{(4)}\ar[lld]& & & \ar[lll]& && \\
  &\ar[uuuuu]& & u^{(1)}\ar[rrr] & & & \ar[rrrd] && \ar[llu]& \\
\ar[urrr] & & & & & & & & & \ar[uuuuuuu]
}
\]

\[\pmb{w}_2: \quad \footnotesize
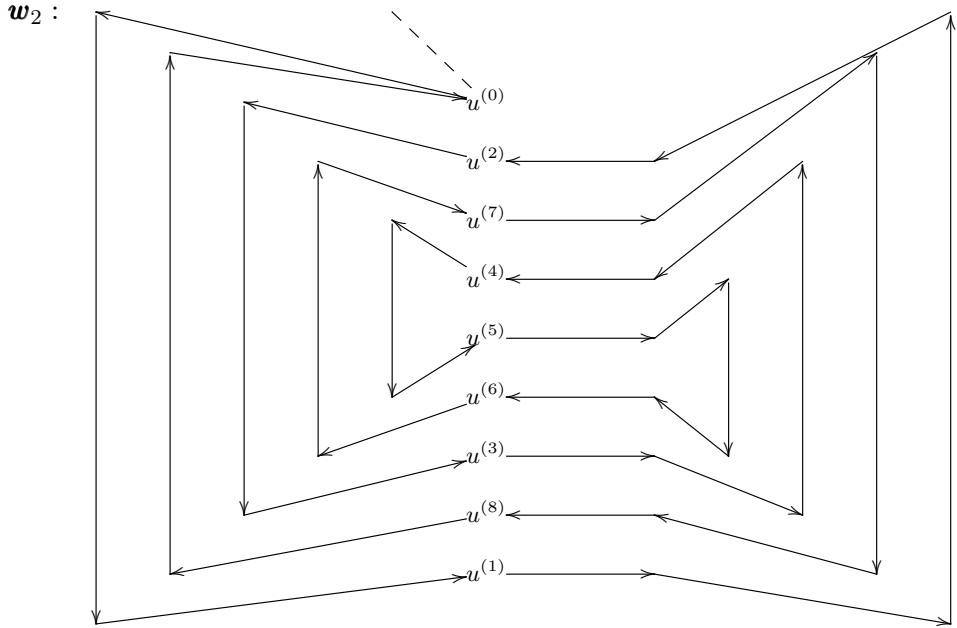
\xymatrix@C=28pt@R=14pt@M=0pt{
\ar[ddddddddddd]& & & & & & & & & & & \ar[llllddd] \\
& \ar[rrrrd]&&&&&&&&& \ar[ddddddddd]& \\
&& \ar[ddddddd]&&& u^{(0)}\ar[llllluu]\ar@{--}[luu]& & & & & & \\
&&& \ar[rrd]&& u^{(2)}\ar[lllu]&& \ar[ll]& & \ar[lldd]&& \\
&&&& \ar[ddd]& u^{(7)}\ar[rr]&& \ar[rrruuu]& & & & \\
&&&&& u^{(4)}\ar[lu]&& \ar[ll]& \ar[ddd]& & & \\
&&&&& u^{(5)}\ar[rr]&& \ar[ru]&&&& \\
&&&& \ar[ru]& u^{(6)}\ar[lld]&& \ar[ll]& & & & \\
&&& \ar[uuuuu]&& u^{(3)}\ar[rr]&& \ar[rrd]& \ar[lu]& & & \\
&& \ar[rrru]&&& u^{(8)}\ar[lllld]&& \ar[ll]&& \ar[uuuuuu]& & \\
& \ar[uuuuuuuuu]&&&& u^{(1)}\ar[rr]&& \ar[rrrrd]&&& \ar[lllu]& \\
\ar[rrrrru]& & & & & & & & & & & \ar[uuuuuuuuuuu]
}
\]

Observe carefully that if one starts with a common vertex not equal to $u$, then one can always expand it to a maximal common subwalk which contains $u$.
Thus, in order to classify maximal common subwalks, and determine the neighbourhoods for which $u$ appears as an intersecting vertex, we start by comparing $u^{(r)}$ with $u^{(s)}$ for some $r,s\in \{ 0,1,\ldots,4p+1 \}$.
A subwalk containing $u=u^{(i)}$ is in one of the following forms:
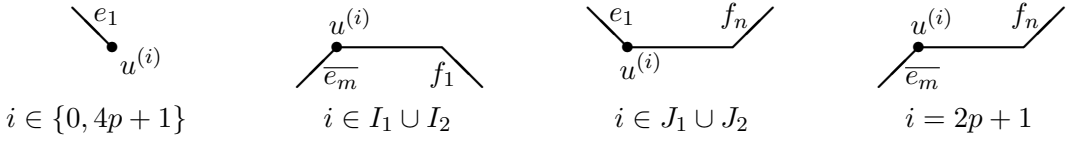
\begin{center}
\begin{picture}(420,65)(0,0)

\thicklines

\put(15,5){$i\in \{ 0,4p+1\}$}
\put(55,35){\line(-1,1){15}}
\put(55,35){\circle*{4}}
\put(58,25){$u^{(i)}$}
\put(48,45){$e_{1}$}

\put(135,5){$i\in I_{1}\cup I_{2}$}
\put(140,35){\circle*{4}}
\put(140,35){\line(-1,-1){15}}
\put(140,35){\line(1,0){40}}
\put(180,35){\line(1,-1){15}}
\put(137,40){$u^{(i)}$}
\put(135,21){$\ol{e_{m}}$}
\put(175,22){$f_{1}$}

\put(245,5){$i\in J_{1}\cup J_{2}$}
\put(250,35){\circle*{4}}
\put(250,35){\line(-1,1){15}}
\put(250,35){\line(1,0){40}}
\put(290,35){\line(1,1){15}}
\put(247,23){$u^{(i)}$}
\put(243,45){$e_{1}$}
\put(285,44){$\ol{f_{n}}$}

\put(355,5){$i=2p+1$}
\put(360,35){\circle*{4}}
\put(360,35){\line(-1,-1){15}}
\put(360,35){\line(1,0){40}}
\put(400,35){\line(1,1){15}}
\put(355,21){$\ol{e_{m}}$}
\put(395,44){$\ol{f_{n}}$}
\put(357,40){$u^{(i)}$}

\end{picture}
\end{center}
where $I_{1}=\{ 1,3,\ldots, 2p-1\}$, $I_{2}=\{ 2p+2,2p+4,\ldots,4p\}$, $J_{1}=\{ 2,4,\ldots,2p\}$, and $J_{2}=\{ 2p+3, 2p+5,\ldots,4p-1 \}$.

The following list shows all possible intersecting vertices and maximal common subwalks obtained by extending a coinciding indexed vertex $u^{(r)}=u^{(s)}$.
\begin{enumerate}
\item \underline{$(r,s)\in \{ 0,4p+1 \}\times (I_{1}\cup I_{2})$}:, $u^{(r)}=u^{(s)}$ specifies an intersecting vertex.
\item \underline{$(r,s)\in \{ 0,4p+1\}\times \{2p+1\}$}: $u^{(r)}=u^{(s)}$ specifies an intersecting vertex.
\item \underline{$r,s\in\{0,4p+1\}$}: $u^{(r)}=u^{(s)}$ gives the coinciding endpoints of $\pmb{W}_p$; the induced maximal common subwalk is given by
\begin{align}
(u^{(0)}=u^{(4p+1)},\pmb{e},\ldots,\pmb{s},u^{(2p)}=u^{(2p+1)}) \in \pmb{w}_p\cap \ol{\pmb{w}_p}.\notag
\end{align}

\item \underline{$(r,s)\in \{0,4p+1\}\times J_{1}$}: The induced maximal common subwalk is given by
\begin{align}
(u^{(0)}=u^{(s)},\pmb{e},\ldots,\pmb{e},u^{(2p+1-s)}=u^{(2p+1)},\pmb{d}) \in \pmb{w}_p\cap \pmb{w}_p & \text{for }r=0;\notag \\
(u^{(4p+1)}=u^{(s)},\pmb{e},\ldots,\pmb{e},u^{(2p-s)}=u^{(2p+1)},\pmb{d}) \in \ol{\pmb{w}_p}\cap \pmb{w}_p & \text{for }r=4p+1;\notag
\end{align}

\item \underline{$(r,s)\in (\{0\}\times J_{2}) \cup (J_1\times \{2p+1\})$}:
For $(r,s)\in \{0\}\times J_2$, the induced maximal common subwalk is given by
\begin{align}
(u^{(0)}=u^{(s)},\pmb{e},\ldots,\pmb{s},u^{(s-2p-1)}=u^{(2p+1)}) \in \pmb{w}_p\cap \ol{\pmb{w}_p}.\notag
\end{align}

Note that $\{s-2p-1\mid s\in J_2\} = J_1\setminus\{2p\}$, so this is also the maximal common subwalk induced by $(r,s) \in (J_{1}\setminus \{2p\})\times \{2p+1\}$.
For $(r,s)=(2p,2p+1)$, observe that the other endpoint of the induced maximal common subwalk is given by $u^{(0)}=u^{(4p+1)}$, i.e. the case of (3) above.

\item \underline{$(r,s)\in (\{4p+1\}\times J_2) \cup (J_2\times \{2p+1\})$}:
For $(r,s)\in J_2 \times \{2p+1\}$, the induced maximal common subwalk is given by
\begin{align}
(u^{(r)}=u^{(2p+1)},\ol{\pmb{s}},\ldots,\ol{\pmb{e}},u^{(4p+1)}=u^{(6p+2-r)}) \in \pmb{w}_p\cap \pmb{w}_p.\notag
\end{align}
Note that $\{ 6p+2-r \mid r\in J_2 \} = J_2$, so this is also the same maximal common subwalk induced by $(r,s)\in \{4p+1\}\times J_2$.

\item \underline{$(r,s)\in (I_{1}\cup I_{2})\times \{2p+1\}$}: 
For $r\in I_1$, the induced maximal common subwalk is given by
\begin{align}
(u^{(0)}=u^{(2p+1-r)},\pmb{e},\ldots,\pmb{e},u^{(r)}=u^{(2p+1)},\pmb{d}) \in \pmb{w}_p\cap \pmb{w}_p.\notag
\end{align}
Similarly, if $r\in I_{2}$ instead, then the corresponding maximal common subwalk which is in $\ol{\pmb{w}_p}\cap \pmb{w}_p$ is given by the same sequence.

\item \underline{$(r,s)\in (I_{1}\cup I_{2})\times (J_{1}\cup J_{2})$}:
For such an $(r,s)$, the induced maximal common subwalk is given by
\begin{align}
(u^{(r)}=u^{(s)}, \pmb{d}).\notag
\end{align}
Note that in the case when $l=0$, $\pmb{d}$ is empty, which means that this maximal common ``subwalk" is just an intersecting vertex $u^{(r)}=u^{(s)}$; one checks easily that (NC3) is satisfied.

\item \underline{$(r,s)\in I_{1}\times I_{1}, J_1\times J_1$}:
When both $r,s$ are in $I_{1}$, assume without loss of generality that $r<s$, then the induced maximal common subwalk is given by
\begin{align}
(u^{(0)}=u^{(s-r)},\pmb{e},\ldots,\pmb{e},u^{(r-s+2p+1)}=u^{(2p+1)},\pmb{d}) \in \pmb{w}_p\cap \pmb{w}_p.\notag
\end{align}
Similarly, if $r,s\in J_{1}$, then the induced maximal common subwalks takes the same form.

\item \underline{$(r,s)\in I_{2}\times I_2, J_2\times J_2$}:
When both $r,s$ are in $I_2$, assume without loss of generality that $r<s$, then the induced maximal common subwalk is given by
\begin{align}
(u^{(2p+1)}=u^{(s-r+2p+1)}, \ol{\pmb{s}},\ldots,\ol{\pmb{e}},u^{(r-s+4p+1)}=u^{(4p+1)}) \in \pmb{w}_p\cap \pmb{w}_p.\notag
\end{align}
Similarly, if $r,s\in J_{2}$, then the induced maximal common subwalks takes the same form.

\item \underline{$(r,s)\in (I_1\times I_2)\cup (J_1\times J_2)$}:
If $r+s>4p+1$, then the induced maximal common subwalk takes the same form as (9); otherwise, it takes the same form as (10).

\end{enumerate}
Using the visualisation above, one can see that the sign and non-crossing conditions are satisfied at each of these cases.
Therefore, $\AW(\G)$ is an infinite set.

This finishes the proof of Proposition \ref{type-odd}.

\section{Homological and Algebraic Preliminaries} \label{sec-alg-hom-prelim}
\subsection{Tilting theory}\label{sec-tilting-theory}
We will work in the bounded homotopy category $\Kb(\proj\Lambda)$ in this subsection.
Without loss of generality, each indecomposable complex takes the form $T=(T^{i},d^{i})_{i\in\Z}$ where $d^i$ lies in the Jacobson radical of $\proj{\Lambda}$ for all $i$.

\begin{definition}
Let $T$ be a complex in $\Kb(\proj\Lambda)$.
\begin{enumerate}
\item We say that $T$ is \emph{pretilting} if it satisfies $\Hom_{\Kb(\proj\Lambda)}(T,T[n])=0$ for all non-zero integers $n$.
\item We say that $T$ is \emph{tilting} if it is pretilting and generates $\Kb(\proj\Lambda)$ by taking direct summands, mapping cones, and shifts.
\item A pretilting complex is said to be \emph{partial} if it is a direct summand of some tilting complex.
\end{enumerate}
We denote by $\tilt\Lambda$\label{sym-tilt} the set of isomorphism classes of basic tilting complexes of $\Lambda$.
\end{definition}

For a complex $T$ in $\Kb(\proj\Lambda)$,
the $n$-th term of $T$ is denoted by $T^n$.
A complex $T$ in $\Kb(\proj\Lambda)$ is called 
\begin{itemize}
\item  \emph{stalk} if it is of the form $(0\to T^n\to 0)$ for some $n\in\Z$;
\item \emph{$n$-term} if it is of the form $(0\to T^{-n+1}\to\cdots \to T^{-1}\to T^{0}\to 0)$ for a non-negative integer $n$. Note that this is different from simply requiring $T$ to be concentrated in $n$ consecutive degrees.
\end{itemize}

The set of basic tilting complexes up to isomorphism has a natural partial order.

\begin{definitiontheorem}\cite{AI,A}
Let $T$ and $U$ be tilting complexes of $\Lambda$.
We write $T\geq U$\label{sym-TgeqU} if $\Hom_{\Kb(\proj\Lambda)}(T,U[n])=0$ for any positive integer $n$.
Then $\geq$ induces a partial order structure on $\tilt\Lambda$.
Moreover, the set $\ntilt{n}\Lambda$\label{sym-ntilt} of $n$-term tilting complexes is precisely the interval $\{T\in\tilt\Lambda\mid\Lambda\geq T\geq \Lambda[n-1] \}$ in $\tilt\Lambda$.
\end{definitiontheorem}

We denote by $\ipt\Lambda$\label{sym-ipt} the set of isomorphism classes of indecomposable two-term pretilting complexes of $\Lambda$.
\begin{proposition}\label{prop-noCommonSummand}\label{b}\cite{AH, A}
Let $\Lambda$ be a symmetric algebra and $T$ be a two-term pretilting complex of $\Lambda$.
Then the following hold:
\begin{enumerate}[{\rm (i)}]
\item $T$ satisfies $\add T^0\cap \add T^{-1}=0$.
\item $T$ is partial tilting.  In fact, it is a direct summand of a two-term tilting complex.
\item It is (two-term) tilting if, and only if, the number of isomorphism classes of indecomposable direct summands of $T$ coincides with that of $\Lambda$ (i.e. $|T|=|\Lambda|$).
\end{enumerate}
\end{proposition}

\subsection{Brauer graph algebras and their modules} \label{sec-BGA}
For convenience, we say that a half-edge $e$ is \emph{truncated} if $\sigma(e)=e$ and $\mm(s(e))=1$.
\begin{definition}\label{def-BGA}
Let $\G$ be a Brauer graph whose multiplicity function is denoted by $\mm$.  

If $\G$ is the connected graph (tree) with one edge $E=\{e,\ol{e}\}$ and two vertices (endpoints of $E$), i.e. $\circ - \circ$, with multiplicity $\mm\equiv 1$, then we define the Brauer tree algebra of $\G$ as $\Bbbk[x]/(x^2)$.
Otherwise, we define the Brauer graph algebra $\Lambda_\G$\label{sym-LambdaG} of $\G$ by giving its quiver $Q_\G$\label{sym-QG} and relations as follows.
\begin{itemize}
\item The set of vertices of $Q_\G$ is defined as the set of edges in $\G$.  For an edge $E=\{e,\ol{e}\}$ in $\G$, we denote the trivial path corresponding to $E$ by $1_E$\label{sym-e0}.
\item If $e$ is not truncated, then there is an arrow from $E':=\{\sigma(e),\ol{\sigma(e)}\}$ to $E=\{e,\ol{e}\}$, denoted by $(e|\sigma(e))$\label{sym-a-e-sig-e}.
\end{itemize}

In general, suppose $e':=\sigma^k(e)$ for some $k\leq\val(s(e))$ such that $e$ is not truncated, we define $(e|e')$\label{sym-a-ef} to be the path $(e|\sigma(e))(\sigma(e)|\sigma^2(e))\cdots(\sigma^{k-1}(e)|e')$.
We call this path \emph{short} if $k\lneqq \val(s(e))$.  
The path $(e|e)$ will be called a \emph{Brauer cycle}.
For a truncated half-edge $e$, we define the Brauer cycle $(e|e)$ to be $(\ol{e}|\ol{e})^{\mm(s(\ol{e}))}$. 
For better readability, sometimes we write $[e]$\label{sym-a-ee} instead of $(e|e)$.
We identify $[e]^0$ and $[\ol{e}]^0$ with $1_{\{e,\ol{e}\}}$.
Note that our notation $(e'|e)$ goes in the opposite direction (from right to left) compare to the one in \cite{KS}.

The relations of $\Lambda_\G$ are generated by the following three types of \emph{Brauer relations}.
\begin{itemize}
\item[(Br1)] If both $e$ and $\ol{e}$ are not truncated, then $[e]^{\mm(s(e))}-[\ol{e}]^{\mm(s(\ol{e}))}=0$.
\item[(Br2)] $(\sigma^{-1}(e)|e)(\ol{e}|\sigma(\ol{e}))=0$ for any $e$ if both $(\sigma^{-1}(e)|e)$ and $(\ol{e}|\sigma(\ol{e}))$ are defined in $Q_\G$.
\item[(Br3)] If $\ol{e}$ is truncated, then  $[e]^{\mm(s(e))}(e|\sigma(e))=0$.
\end{itemize}
\end{definition}
\begin{remark}
Recall our convention of taking right modules and identifying maps with arrows of the Ext-quiver.  It says that a short path $(e'|e)$ can be regarded as a map $P_{\{e,\ol{e}\}}\to P_{\{e',\ol{e'}\}}$ given by multiplying with $(e'|e)$ on the left.
Consequently, we call such a map a \emph{short map}.
\end{remark}

It is well-known that every Brauer graph algebra is a symmetric special biserial algebra and vice versa; see \cite{Sch}.
In particular, an indecomposable non-projective module of a Brauer graph algebra falls into either one of the two sub-classes, called string modules and band modules; this is a result of \cite{WW}, who call them representations of the first and second kind respectively.
For reason to be clear later (cf. Lemma \ref{lem-non-zero-hom}), we will not give any more details about band modules, and concentrate only on the string modules.

The usual convention of defining string modules is to use the so-called string combinatorics of special biserial algebras.
It is not necessary to know string combinatorics in order to understand the statements of our result (Theorem \ref{mainbij}) if we define string modules homologically as follows.
Thus, details about string combinatorics will be left entirely in Section \ref{sec-proof}.
The equivalence between the following definition and the one given by string combinatorics is explained in \cite[Sec 3]{WW}.

\begin{definition}\label{def-string-mod}
An indecomposable non-projective $\Lambda_\G$-module $M$ is a \emph{string module} if its minimal projective presentation $P_M=(P_M^{-1}\xto{d_M}P_M^0)$\label{sym-PM}\label{sym-dM} can be written in one of the following forms:
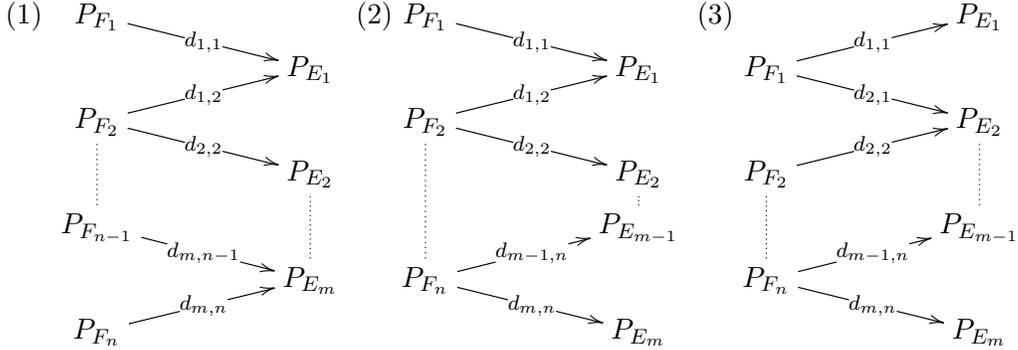
\begin{center}
(1) $\xymatrix@R=3pt@C=50pt{
P_{F_1} \ar[rd]|{d_{1,1}} & \\
 &  P_{E_1}\\
P_{F_2} \ar@{.}[dd]\ar[ru]|{d_{1,2}}\ar[rd]|{d_{2,2}}&\\
 & P_{E_2} \ar@{.}[dd]\\
P_{F_{n-1}}\ar[rd]|{d_{m,n-1}} &  \\
 & P_{E_m} \\
P_{F_n} \ar[ru]|{d_{m,n}} &
}$ (2)$\xymatrix@R=3pt@C=50pt{
P_{F_1} \ar[rd]|{d_{1,1}} & \\
 & P_{E_1}\\
P_{F_2} \ar@{.}[ddd]\ar[ru]|{d_{1,2}}\ar[rd]|{d_{2,2}}&\\
 & P_{E_2} \ar@{.}[d]\\
 & P_{E_{m-1}}\\
P_{F_n}\ar[ru]|{d_{m-1,n}}\ar[rd]|{d_{m,n}} &  \\
 & P_{E_m} 
}$ (3)$\xymatrix@R=3pt@C=50 pt{
 &  P_{E_1}\\
P_{F_1} \ar[rd]|{d_{2,1}} \ar[ru]|{d_{1,1}} & \\
& P_{E_2} \ar@{.}[dd]\\
P_{F_2} \ar@{.}[dd]\ar[ru]|{d_{2,2}}&\\
 & P_{E_{m-1}}\\
P_{F_n}\ar[ru]|{d_{m-1,n}}\ar[rd]|{d_{m,n}} &  \\
 & P_{E_m} }$
\end{center}
with each $d_{i,j}$ given by a left-multiplication of a path from $F_j$ to $E_i$.
If a complex of projective $\Lambda_\G$-modules takes one of the forms above, or is an indecomposable stalk complex concentrated in degree 0 or $-1$, then it is called a \emph{two-term string complex.}
\end{definition}
\begin{remark}\label{rmk-stringmod}
(i) The definition is presented with a ``chosen direction" to keep it short.  In practice (as well as in our forthcoming proofs), one can reorder $E_i$'s and $F_j$'s ``upside-down", i.e. swap $E_i$ with $E_{m-i+1}$ and swap $F_j$ with $F_{n-j+1}$.  Under such reordering, the diagram of (2) is reflected along the horizontal axis in the middle.

(ii) We can write $d_M$ as a bidiagonal matrix $(d_{i,j})_{i,j}$.  Moreover, if the defining path of $d_{i,j}$ is $[e]^k(e|f)$, then it follows from the Brauer relations that the defining path of the non-zero component $d_{i\pm 1,j}$ (resp. $d_{i,j\pm 1}$) is of the form $[e']^{k'}(e'|\ol{f})$ (resp. $[\ol{e}]^{k''}(\ol{e}|f')$).

(iii) We always assume that the minimal projective presentation of a string module takes one of these three forms. 
\end{remark}

The authors of \cite{BM} introduced the notion of homotopy strings and homotopy bands, which they use to parameterise the indecomposable objects in the bounded derived category of a gentle algebra.
One can define an analogue of homotopy strings for a Brauer graph algebra, then two-term string complexes are examples of complex associated to these combinatorial objects.
Since we will only study two-term complexes in this paper, we will drop the adjective ``two-term" for string complexes.

\section{Two-term (pre)tilting complexes via ribbon combinatorics}\label{sec-twotilt}
Let $\G$ be a Brauer graph and $\Lambda:=\Lambda_{\G}$ be the associated Brauer graph algebra.
In this section, we study the relationship between two-term tilting complexes and walks.

Let $T$ be an indecomposable complex in $\Kb(\proj\Lambda)$.
If the homology of $T$ is non-trivial only at two consecutive degrees $n,n+1$ and $T$ is indecomposable, then $T$ is isomorphic to some shifts of a two-term complex given by the minimal projective presentation of an indecomposable $\Lambda$-module $H^{n+1}(T)$.
Dually, $T$ is also isomorphic to some shifts of the minimal injective (co)presentation of an indecomposable $\Lambda$-module $H^n(T)$.
In particular, we will assume that every indecomposable two-term non-stalk complex in $\Kb(\proj\Lambda)$ takes such a form.

\begin{definition}\label{def-maximal-map}
Suppose $T=(T^{-1}\xto{d}T^0)$ is a two-term complex in $\Kb(\proj\Lambda)$ such that $T=P_M$ (hence, $d=d_M=(d_{i,j})_{i,j}$) for a string module $M$. Then we call $d$ (resp. $T$, resp. $M$) a \emph{short string} map (resp. complex, resp. module) if every non-zero $d_{i,j}$ is a short map between indecomposable projective modules.
We denote by $\gcx\Lambda$\label{sym-gcx} the set of indecomposable stalk complexes of projective modules concentrated in degree 0 or $-1$, and two-term short string complexes  $T:=(T^{-1}\xto{d}T^{0})$ which satisfy $\add T^{0}\cap \add T^{-1}=0$.
\end{definition}

Let $w=(e_1,\ldots, e_n; \e)$ be a signed half-walk.
By definition, two consecutive half-edges in $w$ have different signs, 
and so we can define a two-term complex $T_w:=(T^{-1}\xto{d} T^0)$\label{sym-Tw}, where
\begin{align}
T^{-1}:=\bigoplus_{\e(e_{b})=-}P_{\{e_{b},\ol{e_b}\}},\;\; & \;\;
T^{0}:=\bigoplus_{\e(e_a)=+}P_{\{e_{a},\ol{e_a}\}} \notag \\
d=(d_{a,b})_{a,b} \text{ with }& d_{a,b}:=
\begin{cases}
(e_a|\ol{e_{b}})&\text{if }b=a-1;\\
(\ol{e_a}|e_b)&\text{if }b=a+1;\\
0&\textnormal{otherwise.}
\end{cases}\notag
\end{align}
The following properties are almost immediate from the construction.
\begin{lemma}\label{lem-TW-properties}
\begin{itemize}
\item[(1)] $T_{w}=T_{\ol{w}}$ holds.  In particular, for a signed walk $W=\{w,\ol{w}\}$, we can define $T_W:=T_w=T_{\ol{w}}$\label{sym-TW}.
\item[(2)] $T_{W}$ is in $\gcx\Lambda$. 
\end{itemize}
\end{lemma}
\begin{proof}
(1) This is clear by construction. 

(2) This is clear for $w=(e^\pm)$ for any half-edge $e$.  Otherwise, since every $d_{a,b}$ is a short map, the complex $T^{-1}\xto{d}T^0$ is a minimal projective presentation of the (short string) module $H^0(T)$.  
Recall that a signed walk is a walk equipped with a signature.
By the Definition \ref{def-walks} (4) of signature, we must have $\e_W(e_i)=\e_W(\ol{e_i})=\e_W(e_j)=\e_{W}(\ol{e_j})$ for any half-edges $e_i=e_j$ in $W$.
Hence, the follows from the construction of $T_w$ that $\add T^{-1}\cap\add T^0=0$.
\end{proof}

\begin{lemma}\label{lem-SW-biject}
The map $\SW(\G)\to\gcx(\Lambda)$ given by $W\mapsto T_W$ is bijective.
\end{lemma}
\begin{proof}
We prove this by finding the inverse of the map.  We define first a map from $\gcx(\Lambda)\to \SW(\G)$ given by $T\mapsto W_T=\{w_T,\ol{w_T}\}$.  Note that it is sufficient to define just one half-walk $w_T$ along with signatures on the half-edges in $w_T$.

If $T$ is a stalk complex $P_E[s]$ with $s\in\{0,1\}$, then take $w_T=(e)$ with signature being the sign of $(-1)^s$.
If $T=(\oplus_{j=1}^n P_{F_j} \xto{d} \oplus_{i=1}^m P_{E_i})$ is not a stalk complex, then it is a minimal projective presentation of a short string module $M=H^0(T)$.
It follows from Remark \ref{rmk-stringmod}(ii) that we can choose the half-edge representatives $e_1,f_1$ of $E_1,F_1$ respectively, so that $d_{1,1}=(e_1|\ol{f_1})$ when $M$ is of the form (1) or (2) in Definition \ref{def-string-mod}, or $d=(\ol{e_1}|f_1)$ when $M$ is of the form (3) in Definition \ref{def-string-mod}.
Moreover, these determine canonical representatives $e_i,f_j$ so that $d_{1,2}=(\ol{e_1}|f_2)$, $d_{2,2}=(e_2|\ol{f_2})$, $\ldots$, (resp. $d_{2,1}=(e_2|\ol{f_1})$, $d_{2,2}=(\ol{e_2}|f_2)$, $\ldots$) so on and so forth.
We can now write down a half-walk $(f_1,e_1,f_2,\ldots)$ (resp. $(e_1,f_1,e_2,\ldots)$).
Observe that this half-walk is uniquely defined up to applying the involution $\ol{\,\cdot\,}$.
Since $\add T^0\cap \add T^{-1}=0$, one can define a signature $\e$ on $w_T$ by assigning $\e(e_i)=+$ and $\epsilon(f_j)=-$.

Since a short map $(e|\ol{f})$ corresponds uniquely to an ordered pair $(e,\ol{f})$ of distinct half-edges which are incident to the same vertex, the assignment above gives a well-defined injective map.  Lemma \ref{lem-TW-properties} (2) implies that $W\mapsto T_W$ is the inverse of this map.
\end{proof}

The following says that the bijection given above can be improved to give a combinatorial model which describes the indecomposable two-term pretilting complexes.
\begin{lemma}
Every non-stalk indecomposable two-term pretilting complex $T:=(T^{-1}\xto{d}T^{0})$ is a short string complex.
In particular, $\ipt\Lambda$ is a subset of $\gcx\Lambda$.
\end{lemma}

\begin{proof}
Suppose $T^{-1}=\bigoplus_{j=1}^n P_{F_j}$ and $T^0=\bigoplus_{i=1}^m P_{E_i}$.  
Since $T$ is pretilting,
it follows that $\add T^0\cap \add T^{-1}=0$ by Proposition \ref{prop-noCommonSummand}.  In particular, we have $E_i\neq F_j$ for all $i,j$.

Suppose on the contrary that $d_{a,b}:P_{F_b}\to P_{E_a}$ is not short.  As $E_a\neq F_b$, we have a short map $(e_a\,\vline\,\ol{f_b})$.  Let $\alpha:T\to T[1]$ be a map of complexes given by mapping $P_{F_b}$ in $T^{-1}$ to $P_{E_a}$ in $T^0=(T[1])^{-1}$ via $(e_a\,\vline\,\ol{f_b})$.  Non-shortness of $d_{a,b}$ means that we can write it as $[e_a]^k(e_a\,\vline\,\ol{f_b})$ for some $k\geq 1$.  It is easy to see that $\alpha$ is not null-homotopy, or one will have $h\cdot[e_a]^k(e_a\,\vline\,\ol{f_b})=(e_a\,\vline\,\ol{f_b})$ for some $h\in\Hom(P_{\{e_a,\ol{e_a}\}},P_{\{e_a,\ol{e_a}\}})\cong [e_a]^0\Lambda_\G[e_a]^0$, which is not possible.  We now have a non-zero map $\alpha\in\Hom_{\Kb(\proj\Lambda_\G)}(T,T[1])$, contradicting the pretilting-ness of $T$.
\end{proof}

We define one more notion before stating the first main result.
\begin{definition}[Admissible set of signed walks]
A set $\W$ of signed walks is \emph{admissible} if for any pair of (not necessarily distinct) walks $W$ and $W'$ in $\W$, they are non-crossing and satisfy the sign condition.
In particular, admissibility of $W$ is equivalent to that of $\W=\{W\}$.
An admissible set is called \emph{complete} if any admissible set containing $\W$ is $\W$ itself.
Denote by $\CW(\G)$\label{sym-CWG} the set of all complete admissible sets of signed walks.
\end{definition}

For a set $\W$ of signed walks, the map $W\mapsto T_W$ induces a map $\W\mapsto T_\W$, where $T_\W$ is the multiplicity-free (possibly infinite) direct sum of complexes $T_W$ over all $W\in \W$.

\begin{theorem}\label{mainbij}
Let $W,W'$ be signed walks and $T_W, T_{W'}$ be their corresponding two-term complexes (via Lemma \ref{lem-SW-biject}).
Then $T_W\oplus T_{W'}$ is pretilting if and only if $\{W,W'\}$ is admissible.
In particular, 
\begin{itemize}
\item $W\mapsto T_W$ defines a bijection $\AW(\G)\to\ipt\Lambda_\G$, and 
\item $\W\mapsto T_\W$ defines a bijection $\CW(\G)\to\ttilt\Lambda_\G$.
\end{itemize}
\end{theorem}

We give a few easy consequences of the theorem before proving it.

\begin{proposition}\label{2tilt}
Let $\G$ be a Brauer graph.
\begin{enumerate}[(1)]
\item If $\G$ has $n$ edges and $\W$ is a complete admissible set of signed walks of $\G$, then $|\W|=n$.

\item Each edge in $\G$ appears in at least one signed walk in a complete admissible set.

\item If $\G'$ is another Brauer graph with the same underlying ribbon graph as $\G$, then there is an isomorphism between the partially ordered sets $\ttilt{\Lambda_\G}$ and $\ttilt{\Lambda_{\G'}}$.
\end{enumerate}
\end{proposition}
\begin{proof}
(1): By Theorem \ref{mainbij}, $T_\W$ is a tilting complex, which implies $|T_\W|=|\Lambda_\G|=n$.

(2): This follows from the fact that a tilting complex induces a basis of the Grothendieck group $K:=K_0(\Kb(\proj\Lambda_\G))$, and 
$K$ is a free abelian group with canonical basis given by the isoclasses of projective indecomposable $\Lambda_\G$-modules, which is in bijection with edges of $\G$.

(3): The set-wise bijection is immediate from Theorem \ref{mainbij}.  The isomorphism on the partial order structure will be shown at the end of Section \ref{sec-proof}.
\end{proof}

\section{Proof of Theorem \ref{mainbij}}\label{sec-proof}
In this section, we give a proof of Theorem \ref{mainbij}.
Let $T:=(T^{-1}\xto{d}T^{0})$ be a two-term complex in $\Kb(\proj\Lambda)$.
We define two modules $M_{T}$\label{sym-MT} and $N_{T}$\label{sym-NT} as follows:
\begin{align}
M_{T}:= H^{0}(T),\ N_{T}:=H^{-1}(T)=
\begin{cases}
\Omega^2 M_{T}, &\textnormal{if $T^{0}\neq 0$};\\
T^{-1}, &\textnormal{if $T^{0}=0$}.
\end{cases}\notag
\end{align}

The following lemma is the central trick of our proof. 
\begin{lemma}\label{lem-non-zero-hom}
Let $\Lambda$ be a symmetric algebra.
Let $T$ and $T'$ be indecomposable two-term complexes in $\Kb(\proj\Lambda)$.
Then $\Hom_{\Kb(\proj\Lambda)}(T',T[1])=0$ if and only if $\Hom_{\Lambda}(M_{T},N_{T'})=0$.
\end{lemma}
\begin{proof}
Assume that $T'^{0}\neq 0$. Then, by \cite[Lemma 3.4]{AIR}, 
we have that $\Hom_{\Lambda}(M_{T},\Omega^2 M_{T'})=0$ if and only if $\Hom_{\Kb(\proj\Lambda)}(T',T[1])=0$.
Assume that $T'^{0}=0$. Then we have 
\begin{align}
\Hom_{\Lambda}(M_{T}, T'^{-1})=0 
&\Leftrightarrow \Hom_{\Lambda}(T'^{-1}, M_{T})=0\notag\\
&\Leftrightarrow \Hom_{\Kb(\proj\Lambda)}(T',T[1])=0.\notag
\end{align}
Hence, the assertion follows.
\end{proof}

This lemma gives the reason why we are not interested in the band modules.  Indeed, a band module $M$ satisfies $\Omega^2(M)\cong M$, so $\mathrm{id}_M\in \Hom_{\Lambda_\G}(M,\Omega^2 M)=\Hom_{\Lambda_\G}(M,M)$.  In particular, an indecomposable two-term pretilting complex is never isomorphic to the minimal projective presentation of a band module.

Let $T=(T^{-1}\xto{d} T^{0})$ and $T'=(T'^{-1}\xto{d'} T'^{0})$ be complexes in $\gcx\Lambda$.
Let $W:=W_{T}$ and $W':=W_{T'}$ be signed walks with signatures $\e$ and $\e'$ respectively, 
corresponding to $T$ and $T'$ under the bijection in Lemma \ref{lem-SW-biject}.
In order to prove Theorem \ref{mainbij}, we need to understand the interaction between the combinatorics of the pair $W,W'$ and the zeroness of $\Hom_\Lambda(M_T,N_{T'})$.
For this purpose, we recall the combinatorics used to study string modules over special biserial algebras.

\subsection{String combinatorics}\label{subsec-string-com}
Fix a Brauer graph $\G$, and let $A$\label{sym-A} be the algebra $\Lambda_{\G}/\soc(\Lambda_{\G})$.
Then $A$ is given by the bounded path algebra $\Bbbk Q/I$ where $Q=Q_{\G}$ (cf. Definition \ref{def-BGA}) and $I$ is the relational ideal generated by the Brauer relation (Br2) of $\Lambda_{\G}$ along with $[e]^{\mm(s(e))}=0$ for all $e\in H$.

For an arrow $\alpha\in Q_1$, we denote its head by $\hd(\alpha)$\label{sym-head}, and its tail by $\tail(\alpha)$\label{sym-tail}.
Since our convention is to think of arrows as maps, an arrow $\alpha$ will be drawn as $(\hd(\alpha)\xleftarrow{\alpha}\tail(\alpha))$, opposite to the direction used in \cite{BR,Erd}.  
Our convention for head and tails of trivial path is the usual one, i.e. $\hd(1_E)=\tail(1_E)=E$.

We denote by $\alpha^{-1}$\label{sym-inverse} the (formal) \emph{inverse} of $\alpha$, and by $Q_1^{-1}$ the set of formal inverses.
We also set $\hd(\alpha^{-1}):=\tail(\alpha)$, $\tail(\alpha^{-1})=\hd(\alpha)$, $(\alpha^{-1})^{-1}:=\alpha$, and $1_E^{-1}=1_E$.
The set $\mathcal{A}$\label{sym-alphabet-set} of \emph{alphabets} associated to $A$ consists of elements $1_E, \alpha,\alpha^{-1}$ over all $E\in Q_0$ and all $\alpha\in Q_1$.
An alphabet is \emph{directed} (resp. \emph{inverse}) if it is a trivial path or is in $Q_1$ (resp. $Q_1^{-1}$).
We also call $1_E$ the \emph{trivial alphabet at $E$} for any $E\in Q_0$.

A \emph{word} $\word$\label{sym-word} is a sequence $\alpha_1\alpha_2\cdots \alpha_l$ of alphabets $\alpha_i\in \mathcal{A}$ so that $\tail(\alpha_{i})=\hd(\alpha_{i+1})$ for $i\in\{1,2,\ldots,l-1\}$.  
The head and tail of a word are defined by $\hd(\word):=\hd(\alpha_1)$, $\tail(\word):=\tail(\alpha_l)$.
A \emph{subword} of $\word$ is just a continuous subsequence of $\word$.
The \emph{inverse} of a word is $\word^{-1}:=\alpha_l^{-1}\alpha_{l-1}^{-1}\cdots \alpha_1^{-1}$.
For example, if $e,f$ are half-edges with $s(e)=s(f)$, then the path $(e|f)=(e|\sigma(e))(\sigma(e)|\sigma^2(e))\cdots (\sigma^{-1}(f)|f)$ can be regarded as a word with $\hd(e|f)=E$ and $\tail(e|f)=F$.

If $\word=\alpha_1\alpha_2\cdots \alpha_k$ and $\word'=\alpha_1'\alpha_2'\cdots \alpha_l'$ are words with $\tail(\word)=\hd(\word')$, then $\word\word'$ is the \emph{concatenation} of $\word$ and $\word'$ given by $\alpha_1\alpha_2\cdots \alpha_k\alpha_1'\alpha_2'\cdots \alpha_l'$.
Consider the set $\mathcal{S}$ given by words of the form $\word=\alpha_1\alpha_2\cdots\alpha_k$ such that $\alpha_i\neq \alpha_{i+1}^{-1}$ for all $i\in\{1,\ldots,l-1\}$, and no subword or its inverse belongs to the relational ideal $I$.
We can define an equivalence relation $\sim$ on $\mathcal{S}$ generated by $\word\sim\word$, $\word^{-1}\sim\word$, and $\word 1_{E}\word'\sim\word\word'$ for $E=\hd(\word')=\tail(\word)$.
A \emph{string} $\word=\alpha_1\alpha_2\cdots\alpha_l$ (of $A$) is a representative in an equivalence class of $\mathcal{S}$.
It is said to be \emph{directed} (resp. \emph{inverse}) if all  $\alpha_i$'s are directed (resp. inverse).
We say that $\word$ is a \emph{substring} of $\word'$ if there exist (possibly trivial) strings $\mathsf{u},\mathsf{v}$ such that the concatenation $\mathsf{u}\word\mathsf{v}$ is equivalent to $\word'$.

\medskip

For a trivial string $\word=1_E$, the associated string module, denoted by $M(\word)$, is the simple module $S_E$ corresponding to $E\in Q_0$.
If $\word=\alpha_1\alpha_2\cdots \alpha_l$ is a non-trivial string, then there is a quiver $Q_\word$\label{sym-Qword} whose underlying graph is a line with vertex set $\{0,1,\ldots,n\}$, and arrows $(i\from i+1)$ if $\alpha_{i+1}\in Q_1$; $(i\to i+1)$ otherwise.
There is a representation $V$ of $Q_\word$ given by putting a 1-dimensional $\Bbbk$-vector space at each vertex of $Q_{\word}$ and the identity map on each arrow.
By construction, there is a canonical morphism of quivers $\phi:Q_{\word}\to Q$ which respects the relations of $A$, i.e. the image of a non-zero path in $Q_{\word}$ does not belong to the relational ideal $I$ of $A$.
In particular, we have a functor $\mod{\Bbbk Q_{\word}}\to \mod{A}$ which sends the indecomposable representation $V$ to an indecomposable $A$-module $M(\word)$.
We call $M(\word)$ the \emph{string module $M(\word)$ associated to $\word$}.
Note that $\word\sim \word'$ is equivalent to $M(\word)= M(\word')$.

By the Drozd-Kirichenko rejection lemma \cite{DK}, the canonical (fully faithful) embedding $\mod{A}\to \mod{\Lambda_{\G}}$ induces a bijection between the set of isomorphism classes of indecomposable $A$-modules to the set of isomorphism classes of indecomposable non-projective $\Lambda_{\G}$-modules.
Hence, $M(\word)$ can be regarded as an indecomposable (non-projective) $\Lambda_{\G}$-module naturally for any string $\word$ of $A$.

\begin{example}\label{eg-hook}
(1)  For each $e\in H$, we define a directed string 
\[
\eta_e := \begin{cases}
[e]^{\mm(s(e))-1}(e|\sigma^{-1}(e)), & \text{if $e$ is not truncated;}\\
1_{\{e,\ol{e}\}}, & \text{if $e$ is truncated.}
\end{cases}
\]\label{sym-hook}
The associated string module $M(\eta_e)$ is isomorphic to what Roggenkamp calls maximal (and co-maximal) uniserial non-projective module \cite{Ro}.
Its minimal projective presentation is given by $P_{\{\sigma(\ol{e}),\ol{\sigma(\ol{e})}\}}\xto{(\ol{e}\,|\,\sigma(\ol{e}))\cdot-} P_{\{e,\ol{e}\}}$, which is associated to the signed half-walk $(e^+,\sigma(\ol{e})^-)$.

(2)  For an edge $E=\{e,\ol{e}\}$ of $\G$ (i.e. a vertex in $Q_0$), define two strings
\begin{align}
\word_E := \eta_e\eta_{\ol{e}}^{-1} & \text{ and }\word^E := \eta_e^{-1}\eta_{\ol{e}}.\notag 
\end{align}
Then $M(\word_E)\cong \mathsf{rad}P_E$ and $M(\word^E)\cong P_E/\soc P_E$ as $\Lambda_\G$-modules respectively.
\end{example}

As we have mentioned previously, the minimal projective presentation of a string module (in the sense above) takes the form of Definition \ref{def-string-mod}; for a detailed explanation, see \cite[Section 3]{WW}.

\medskip

Let $T,T'$ be the complexes associated to half-walks $w,w'$ as before.
We are now going to write down the strings $\word, \word'$ so that 
\begin{align}
M(\word) & \cong \begin{cases}
M_T, & \text{if }M_T\notin \proj\Lambda;\\
M_T/\soc M_T, & \text{otherwise,} 
\end{cases} \notag \\
\text{and }M(\word') & \cong \begin{cases}
N_{T'}, & \text{if }N_{T'}\notin \proj\Lambda;\\
\mathsf{rad}N_{T'}, & \text{otherwise.} 
\end{cases}\notag
\end{align}
In particular, we have $\Hom_{\Lambda_{\G}}(M_T,N_{T'})\cong \Hom_A(M(\word),M(\word'))$.
For convenience, we call $\word,\word'$ the \emph{strings associated to $w,w'$} respectively.

To avoid being too repetitive, we change the convention on the indices of the half-edges in a half-walk temporarily as follows.
Let $w$ be a signed half-walk defined by
\begin{align}\label{eq-w-alt}
w:=( (e_0^+,)\; e_1^-, e_2^+, \ldots, e_m^- \;(, e_{m+1}^+)).
\end{align}
The bracket terms let us consider the three different possible half-walks (i.e. starting and ending in the a positively-signed half-edge; starting and ending in a negatively-singed half-edge; starting and ending in half-edges with different signs) in one setting by removing one, or both, of $e_0$ and $e_{m+1}$.
In particular, in the case when neither $e_0$ nor $e_{m+1}$ is removed, the virtual half-edges attached to the endpoints of $w$ are enumerated by $\ol{e_{-1}}$ and $e_{m+2}$; similarly in the case when one of $e_0$ and $e_{m+1}$ is removed.

When $w=(e_0^+)$, $M_T$ is just the indecomposable projective module $P=P_{\{e_0,\ol{e_0}\}}$, and we take $\word = \word^{\{e_0,\ol{e_0}\}}=\eta_{e_0}^{-1}\eta_{\ol{e_0}}$ (which gives $M(\word)\cong P/\soc P$).
When $w=(e_1^-)$, $M_T$ is zero, and so we take $\word$ to be an empty string.
In all other cases (i.e. $M_T$ is non-projective), we define $\word=\word_1\word_2\word_3$, where $\word_i$'s are given as follows.
\begin{align}
\word_1 & = \begin{cases}
1_{\{e_2,\ol{e_2}\}}, & \text{if $e_0$ is removed and $\sigma^{-1}(\ol{e_1})=e_2$;} \\
(e_2|\sigma^{-1}(\ol{e_1}))^{-1}, & \text{if $e_0$ is removed and $\sigma^{-1}(\ol{e_1})\neq e_2$;} \\
\eta_{e_0}^{-1}(\ol{e_0}|e_1)(e_2|\ol{e_1})^{-1}, & \text{if $e_0$ is not removed.}
\end{cases} \notag \\
\word_2 & = (\ol{e_2}|e_3) (e_4|\ol{e_3})^{-1} \cdots (e_{k-1}|\ol{e_{k-2}})^{-1}. \notag \\
\word_3 & = \begin{cases}
1_{\{e_{m-1},\ol{e_{m-1}}\}}, & \text{if $e_{m+1}$ is removed and $\sigma^{-1}(e_m)=\ol{e_{m-1}}$;} \\
(\ol{e_{m-1}}|\sigma^{-1}(\ol{e_m})), & \text{if $e_{m+1}$ is removed and $\sigma^{-1}(e_m)\neq \ol{e_{m-1}}$;} \\
(\ol{e_{m-1}}|e_m)(e_{m+1}|\ol{e_m})^{-1}\eta_{\ol{e_{m+1}}}, & \text{if $e_{m+1}$ is not removed.} \notag 
\end{cases}
\end{align}

Using \cite[Section 3]{WW}, one can write down the minimal projective presentation of $M(\word)$ and see that it coincides with $T_w$, i.e. we have $M(\word)\cong M_T$.

\medskip 

Adopting similar conventions, we take
\begin{align}\label{eq-w'-alt}
w':=((e_0'^{-},)\; e_1'^+, e_2'^-, \ldots, e_n'^+\; (, e_{n+1}'^-)).
\end{align}
When $w=(e_0'{}^-)$, $N_{T'}$ is just the indecomposable projective module $P=P_{\{e_0,\ol{e_0}\}}$, and we take $\word' = \eta_{\sigma(e_0)}^{-1}\eta_{\sigma(\ol{e_0})}$ (which gives $M(\word)\cong \mathsf{rad} P$).
When $w=(e_1'{}^+)$, $N_{T'}$ is zero, and so we take $\word$ to be an empty string.
In all other cases (i.e. $N_{T'}$ is non-projective), we define $\word'=\word'_1\word'_2\word'_3$, where each $\word'_i$ is defined as follows.
\begin{align}
\word'_1 & = \begin{cases}
1_{\{e_2',\ol{e_2'}\}}, & \text{if $e_0$ is removed and $\sigma(\ol{e_1'})=e_2'$;} \\
(\sigma(\ol{e_1'}|e_2')), & \text{if $e_0$ is removed and $\sigma(\ol{e_1'})\neq e_2'$;} \\
\eta_{\sigma(e_0')}(e_1'|\ol{e_0'})^{-1}(\ol{e_1'}|e_2'), & \text{if $e_0'$ is not removed.}
\end{cases} \notag \\
\word'_2 & = (e_3'|\ol{e_2'})^{-1}(\ol{e_3'}|e_4')  \cdots (\ol{e_{n-2}'}|e_{n-1}'). \notag \\
\word'_3 & = \begin{cases}
1_{\{e_{n-1}',\ol{e_{n-1}'}\}}, & \text{if $e_{n+1}'$ is removed and $\sigma(e_n')=\ol{e_{n-1}'}$;} \\
(\sigma(e_n')|\ol{e_{n-1}'})^{-1}, & \text{if $e_{n+1}'$ is removed and $\sigma(e_n')\neq \ol{e_{n-1}'}$;} \\
(e_n'|\ol{e_{n-1}'})^{-1}(\ol{e_n'}|e_{n+1}')\eta_{\sigma(\ol{e_{n+1}'})}^{-1}, & \text{if $e_{n+1}'$ is not removed.} \notag 
\end{cases}
\end{align}
One can write down the minimal injective copresentation of $M(\word')$ and see that it coincides with $T'$.
Alternatively, one can check that $M(\word')\cong N_{T'}$ in the following way.
Write down the string defining $H^0(T')$ using the formulae in the previous part, then one can obtain the string associated to the Auslander-Reiten translate of $H^0(T')$ using standard tricks in string combinatorics; see, for example, \cite[II.6]{Erd}.
The resulting string is then equivalent to $\word'$, because $\Lambda_\G$ being symmetric implies that the Auslander-Reiten translate of $H^0(T')$ is isomorphic to $\Omega^2(H^0(T'))\cong H^{-1}(T')$.

\begin{example}\label{eg-loewy0}
Consider the ribbon graph $\G$ and signed half-walk \[
w_1=(1^+,2^-,3^+,\ol{6^-},\ol{5^+},\ol{4^-},\ol{3^+},\ol{2^-},\ol{1^+},4^-)=(e_1^+,e_2^-,\ldots, e_{10}^-)
\]
in Example \ref{eg-walk1}.
Let $\word=\word_1\word_2\word_3$ be the string constructed using the above algorithm so that $M(\word)\cong M_{T}:=H^0(T_{w_1})$.
We get that 
\[
\word_1 = \eta_1^{-1}(\ol{1}|2)(3|\ol{2})^{-1}\;\;, \;\; \word_2 = (\ol{3}|\ol{6})(\ol{5}|6)^{-1}(5|\ol{4})(\ol{3}|4)^{-1}(3|\ol{2})(\ol{1}|2)^{-1}\;\;,\;\; \word_3 = (1|\ol{3}).
\]
Note that the exact form of $\eta_1$ depends on the choice of multiplicity one equips on $\G$.
Similarly, for the string $\word'=\word_1'\word_2'\word_3'$ associated to $N_{T'}:=H^{-1}(T_{w_1})$, we have
\[
\word_1' = 1_{2,\ol{2}}\;\;,\;\; \word_2'= (3|\ol{2})^{-1}(\ol{3}|\ol{6})(\ol{5}|6)^{-1}(5|\ol{4})(\ol{3}|4)^{-1}(3|\ol{2})\;\;,\;\;  \word_3'=(\ol{1}|2)^{-1}(1|4)\eta_{5}^{-1} .
\]
\end{example}

\subsection{Analysing the Hom-space}
As a consequence of the main result in \cite{CB}, the Hom-space between two string modules can be  described easily.
Instead of stating it in its original form, we present this result in the language of string combinatorics.
\begin{theorem}{\rm \cite{CB}}\label{thm-CB}
For any strings $\word,\word'$ of $A$, the Hom-space $\Hom_A(M(\word),M(\word'))$ admits a $\Bbbk$-basis $\{f_{\mathsf{u}}\}_{\mathsf{u}}$ indexed by strings $\mathsf{u}$ which satisfy the following conditions.
\begin{itemize}
\item[(i)] $\mathsf{u}$ is a substring of both $\word$ and $\word'$.
\item[(ii)] For any non-trivial alphabet $\alpha$, $\alpha \mathsf{u}$ (resp. $\mathsf{u}\alpha$) is a substring of $\word$ implies that $\alpha$ is inverse (resp. directed).
\item[(iii)] For any non-trivial alphabet $\alpha$, $\alpha \mathsf{u}$ (resp. $\mathsf{u}\alpha$) is a substring of $\word'$ implies that $\alpha$ is directed (resp. inverse).
\end{itemize}
\end{theorem}
Note that the string module $M(\mathsf{u})$ is the image of $f_{\mathsf{u}}$.
In particular, the condition (ii) (resp. (iii)) is equivalent to saying that $M(\mathsf{u})$ is a quotient of $M(\word)$ (resp. a submodule of $M(\word')$).

\begin{lemma}\label{lem-M-struc}
Let $w$ be a signed half-walk $( (e_0^+,)\; e_1^-, e_2^+, \ldots, e_m^- \;(, e_{m+1}^+))$, where the convention of the bracketed terms are as in (\ref{eq-w-alt}).
Let $\word$ be the string associated to $w$.
Suppose $\mathsf{u}=\mathsf{u}_1\mathsf{u}_2\cdots \mathsf{u}_k$ is a subword of $\word$ with even $k\geq 2$ so that $\mathsf{u}_i$ is directed (resp. inverse) for all even (resp. odd) $i$.
\begin{enumerate}
\item[(1)] Assume that both $[e_0]$ and $([e_{m+1}])^{-1}$ are not subwords of $\mathsf{u}$.
Then $M(\mathsf{u})$ is a quotient of $M(\word)$ if, and only if, there exists an integer $a\geq 0$ such that the following conditions are satisfied. 
\begin{itemize}
\item Both $e_{2a}, e_{2a+k-2}$ are not virtual.

\item $\mathsf{u}_1$ is $1_{\{e_{2a},\ol{e_{2a}}\}}$, or $(e_{2a}|f)^{-1}$ for some $f\in H$ with cyclic subordering $(e_{2a},f,\ol{e_{2a-1}})_{s(e_{2a})}$.

\item $\mathsf{u}_k$ is $1_{\{e_{2a+k-2},\ol{e_{2a+k-2}}\}}$, or $(\ol{e_{2a+k-2}}|g)$ for some $g\in H$ with cyclic subordering $(\ol{e_{2a+k-2}},g,e_{2a+k-1})_{s(\ol{e_{2a+k-2}})}$.
\end{itemize}
Moreover, in such a case, we have $\mathsf{u}_i=(\ol{e_{2a+i-2}}|e_{2a+i-1})$ for all even $i\in \{2,4,\ldots,k-2\}$ and $\mathsf{u}_i=(e_{2a+i-1}|\ol{e_{2a+i-2}})^{-1}$ for all odd $i\in\{3,5,\ldots, k-1\}$.

\item[(2)] If $M(\mathsf{u})$ is a quotient of $M(\word)$, then there exists a subword of $\mathsf{u}$ which satisfies the conditions in (1).
\end{enumerate}
\end{lemma}
\begin{proof}
Apply the criteria of Theorem \ref{thm-CB} (ii) to $\mathsf{u}$ and $\word$, then it is clear that $\mathsf{u}_i$ is given as stated for all $i\in\{2,3,\ldots, k-1\}$, whereas $\mathsf{u}_1$ is given by the concatenation of $[e_{2a}]^s$ with $s\geq 0$ and the string $(e_{2a}|f)^{-1}$ stated in (1).
Moreover, we have 
\[
s\geq 1 \Leftrightarrow  \mm(s(e_{2a}))>1 \text{ and }a = \begin{cases}
2, &\text{if $e_0$ is removed;}\\ 0, &\text{otherwise.}
\end{cases}
\]
In such a case, removing $[e_{2a}]^s$ from $\mathsf{u}_1$ gives us a new substring which still satisfies Theorem \ref{thm-CB} (ii); hence, defines a quotient of $M(\word)$.
The claims now follow from applying similar reasoning for $\mathsf{u}_k$.
\end{proof}

Dually, one can deduce the following result from Theorem \ref{thm-CB} (iii).

\begin{lemma}\label{lem-N-struc}
Let $w$ be a signed half-walk $((e_0'^{-},)\; e_1'^+, e_2'^-, \ldots, e_n'^+\; (, e_{n+1}'^-))$, where the convention of the bracketed terms are as in (\ref{eq-w'-alt}).
Let $\word'$ be the string associated to $w'$.
Suppose $\mathsf{u}'=\mathsf{u}'_1\mathsf{u}'_2\cdots \mathsf{u}'_l$ is a subword of $\word'$ with even $l\geq 2$ so that $\mathsf{u}_i$ is inverse (resp. directed) for all even (resp. odd) $i$.
\begin{enumerate}
\item[(1)] Assume that both $[e_0']$ and $([e_{n+1}'])^{-1}$ are not subwords of $\mathsf{u}$.
Then $M(\mathsf{u}')$ is a submodule of $M(\word')$ if, and only if, there exists an integer $b\geq 0$ such that the following conditions are satisfied. 
\begin{itemize}
\item Both $e_{2b}', e_{2b+l}'$ are not virtual.

\item $\mathsf{u}'_1$ is $1_{\{e_{2b}',\ol{e_{2b}'}\}}$, or $(f'|e_{2b}')$ for some $f'\in H\setminus\{ e_{2b}', \ol{e_{2b-1}'}\}$ with cyclic subordering $(e_{2b}',\ol{e_{2b-1}'},f')_{s(e_{2b}')}$.

\item $\mathsf{u}'_l$ is $1_{\{e_{2b+l-2}',\ol{e_{2b+l-2}'}\}}$, or $(g'|\ol{e_{2b+l-2}'})^{-1}$ for some  $g'\in H \setminus \{\ol{e_{2b+l-2}'},e_{2b+l-1}'\}$ with cyclic subordering $(\ol{e_{2b+l-2}'},e_{2b+l-1}',g')_{s(\ol{e_{2b+l-2}'})}$.
\end{itemize}
Moreover, in such a case, we have $\mathsf{u}'_i=(e_{2b+i-1}'|\ol{e_{2b+i-2}'})^{-1}$ for all even $i\in\{2,4,\ldots, l-2\}$ and $\mathsf{u}'_i=(\ol{e_{2b+i-2}'}|e_{2b+i-1})$ for all odd $i\in\{3,5,\ldots, l-1\}$.

\item[(2)] If $M(\mathsf{u}')$ is a submodule of $M(\word')$, then there exists a subword of $\mathsf{u}'$ which satisfies the conditions in (1).
\end{enumerate}
\end{lemma}

\begin{lemma}\label{lem-L-uniserial}
Consider the $\Lambda_{\G}$-modules $M_T:=H^0(T_W)$ and $N_{T'}:=H^{-1}(T_{W'})$, where $W,W'$ are signed walks on $\G$.
The following are equivalent:
\begin{itemize}
\item[(U1)] There is a non-zero map $\alpha:M_T\to N_{T'}$ which factors through a uniserial module.

\item[(U2)] There are $w=(e_1,\ldots,e_m;\e)\in W$ and $w'=(e'_1,\ldots,e'_n;\e')\in W'$, $i\in\{1,2,\ldots,m\}$, and $j\in\{1,2,\ldots,n\}$ such that $\e(e_i)=+$, $\e'(e_j')=-$, $s(e_i)=s(e_j')$, and one of the following conditions holds:
\begin{itemize}
\item[(i)] $e_i=e_j'$.
\item[(ii)] $(e_i,e_j',\ol{e_{i-1}}=\ol{e_{j-1}'})_v$ is a cyclic subordering around $v$.
\item[(iii)] $(e_i,e_j',\ol{e_{j-1}'},\ol{e_{i-1}})_v$ is a cyclic subordering around $v$.
\item[(iv)] $(e_i,e_j',\ol{e_{i-1}},\ol{e_{j-1}'})_v$ is a cyclic subordering around $v$.
\end{itemize}
\end{itemize}
\end{lemma}
\begin{proof}
Let $\mathsf{u},\mathsf{u}'$ be subwords of $\word,\word'$ respectively so that the uniserial image of the non-zero map $\alpha$ is $M(\mathsf{u})\cong M(\mathsf{u}')$.
Then we have $\mathsf{u}'=\mathsf{u}^\delta$ for some $\delta\in\{+1,-1\}$.

By Lemma \ref{lem-M-struc}, there is some $i\in\{1,2,\ldots, m\}$ with $\e(e_i)=+$ so that $\mathsf{u}$ is one of the following.
\begin{itemize}
\item[(M1)] $1_{\{e_i,\ol{e_i}\}}$.
\item[(M2)] $(e_i|f)^{-1}$ for some $f\in H\setminus\{e_i,\ol{e_{i-1}}\}$ with cyclic subodering $(e_i,f,\ol{e_{i-1}})_{s(e_i)}$.
\item[(M3)] $(\ol{e_i}|g)$ for some $g\in H\setminus\{\ol{e_i},e_{i+1}\}$ with cyclic subordering $(\ol{e_i},g,e_{i+1})_{s(\ol{e_i})}$.
\end{itemize}

Dually, by Lemma \ref{lem-N-struc}, there is some $j\in\{1,2,\ldots, n\}$ with $\e'(e_j')=-$ so that $\mathsf{u'}$ is one of the following.
\begin{itemize}
\item[(N1)] $1_{\{e_j',\ol{e_j'}\}}$.
\item[(N2)] $(f'|e_j')$ for some $f'\in H\setminus\{e_j',\ol{e_{j-1}'}\}$ with cyclic subordering $(e_j',\ol{e_{j-1}'},f')_{s(e_j')}$.
\item[(N3)] $(g'|\ol{e_j'})^{-1}$ for some $g'\in H\setminus\{\ol{e_j'}, e_{j+1}'\}$ with cyclic subordering $(\ol{e_j'}, e_{j+1}', g')_{s(\ol{e_j'})}$.
\end{itemize}

If $\mathsf{u}^{-1}=\mathsf{u}'$, then $(\mathsf{u},\mathsf{u}')$ is in exactly one of the following forms.
\begin{itemize}
\item \underline{((M1),(N1))}: We get that $\mathsf{u}=1_E$ with $E=\{e_i,\ol{e_i}\}=\{e_j,\ol{e_j'}\}$ (i.e. case (i)).

\item \underline{((M2),(N2))}:
Now we have $\mathsf{u}^{-1}=(e_i|e_j')$ with $e_j'\neq e_i$, and the cyclic subordering around $v=s(e_i)=s(e_j')$ satisfies one of the cases (ii), (iii), (iv).

\item \underline{((M3),(N3))}: We go through the whole argument from the beginning again after replacing $w,w'$ by $\ol{w},\ol{w'}$ respectively.  In this new setting, $(\mathsf{u},\mathsf{u}')$ is will be in form of ((M2),(N2)).
\end{itemize}

If $\mathsf{u}=\mathsf{u}'$, then we can replace $w$ by $\ol{w}$ and apply the arguments used for $\mathsf{u}^{-1}=\mathsf{u}'$.

Conversely, given one of the conditions in (i) to (iv), we can construct $\mathsf{u},\mathsf{u}'$ in the form described above, then it follows from Lemma \ref{lem-M-struc} and Lemma \ref{lem-N-struc} that there exists a non-zero map $f_{\mathsf{u}}: M_T\to N_{T'}$.
\end{proof}

\begin{example}\label{eg-loewy1}
Consider again the ribbon graph $\G$ and signed half-walk $w_1$ of Example \ref{eg-loewy0}.
Let $\G_0$ be a Brauer graph given by equipping $\G$ with constant multiplicity $\mm\equiv 1$, and $T=T'$ be the string complex associated to $w_1$.
Recall that we have strings
\begin{align}
\word&= (1|\ol{6})^{-1}(\ol{1}|2)(3|\ol{2})^{-1}\underline{(\ol{3}|\ol{6})}(\ol{5}|6)^{-1}(5|\ol{4})(\ol{3}|4)^{-1}(3|\ol{2})(\ol{1}|2)^{-1}(1|\ol{3}), \notag \\
\word'&= (3|\ol{2})^{-1}(\ol{3}|\ol{6})(\ol{5}|6)^{-1}(5|\ol{4})(\ol{3}|4)^{-1}(3|\ol{2})(\ol{1}|2)^{-1}\underline{(1|4)}(5|\ol{4})^{-1}, \notag 
\end{align}
so that $M_T\cong M(\word)$ and $N_{T'}\cong M(\word')$.
The underlined part of $\word$ (resp. $\word'$) is given by $(\ol{3}|4)(4|\ol{6})$ (resp. $(1|\ol{3})(\ol{3}|4)$).
Therefore, $(\ol{3}|4)$ satisfies the conditions of Theorem \ref{thm-CB}, and $M\big((\ol{3}|4)\big)$ defines a quotient of $M_T$ as well as a submodule of $N_{T'}$.
By (the proof of) Lemma \ref{lem-L-uniserial}, this corresponds to the cyclic subordering $(\ol{e_3},e_{10},e_4,\ol{e_9})_v=(\ol{3},4,\ol{6},1)_v$ around $v$, where (U2)(iv) is satisfied (taking $w=\ol{w_1}, w'=w_1, i=8, j=10$).
\end{example}

\begin{lemma}\label{lem-L-caseC}
The following are equivalent:
\begin{itemize}
\item[(L1)] There is a non-zero map $\alpha:M_T\to N_{T'}$ which factors through a non-uniserial module.
\item[(L2)] For some $w=(e_1,\ldots,e_m;\e)\in W$,  $w'=(e_1',\ldots,e_n';\e')\in W'$, $i\in\{1,\ldots,m\}$, and $j\in\{1,\ldots,n\}$, there is a signed half-walk $z\in w\cap w'$ given by $z=(t_1,\ldots,t_\ell)$ which satisfies 
\begin{itemize}
\item[\textbullet] $t_k=e_{i+k-1}=e_{j+k-1}'$ for all $k\in\{1,\ldots,\ell\}$, 
\item[\textbullet] $\e(t_k)=\e'(t_k)$ for all $k\in\{1,\ldots,\ell\}$, and
\item[\textbullet] the neighbourhood cyclic orderings are
$(\ol{e_{j-1}'},\ol{e_{i-1}},t_1)_{s(t_1)}$ and $(\ol{t_\ell},e_{j+\ell}',e_{i+\ell})_{s(\ol{t_\ell})}$.
\end{itemize}
\end{itemize}
The condition (L2) means that the two walks can be visualised locally as the solid lines in the pictures below:
\begin{align}\label{dgm-NC}
\xymatrix@R=20pt@C=20pt{
&\ar@{--}[rd]^{\ol{e_{j-1}'}}&&&&&\\
W:&&u\ar@{-}[r]^{t_1}&\ar@{.}[r]&\ar@{-}[r]^{\ol{t_\ell}}&v\ar@{--}[rd]_{e_{j+\ell}'}\ar@{-}[ru]^{e_{i+\ell}}&\\
&\ar@{-}[ru]_{\ol{e_{i-1}}}&&&&&
}\hspace{5mm}\xymatrix@R=20pt@C=20pt{
&\ar@{-}[rd]^{\ol{e_{j-1}'}}&&&&&\\
W':&&u\ar@{-}[r]^{t_1}&\ar@{.}[r]&\ar@{-}[r]^{\ol{t_\ell}}&v\ar@{-}[rd]_{e_{j+\ell}'}\ar@{--}[ru]^{e_{i+\ell}}&\\
&\ar@{--}[ru]_{\ol{e_{i-1}}}&&&&&
}
\end{align}
where $u=s(t_1)$, $v=s(\ol{t_\ell})$.
\end{lemma}
\begin{proof}
Similar to the proof of Lemma \ref{lem-L-uniserial}, we take $\mathsf{u}$ and $\mathsf{u}'$ to be the respective subwords of $\word,\word'$ defining the image of $\alpha$.

By Lemma \ref{lem-M-struc}, there are $i\in \{ 1,2,\ldots, m \}$ with $\e(e_{i})=+$ and $k\geq 1$ such that 
\begin{align}
\mathsf{u}= \mathsf{s} (\ol{e_i}|e_{i+1})\cdots (e_{i+k}|\ol{e_{i+k-1}})^{-1} \mathsf{t}, \notag
\end{align}
where $\mathsf{s}=1_{\{e_i,\ol{e_i}\}}$ or $(e_i|f)^{-1}$ with cyclic subordering $(e_{i}, f, \ol{e_{i-1}})_{s(e_{i})}$, and $\mathsf{t}=1_{\{e_{i+k},\ol{e_{i+k}}\}}$ or $(\ol{e_{i+k}}|g)$ with cyclic subordering $(\ol{e_{i+k}}, g, e_{i+k+1})_{s(\ol{e_{i+k}})}$.

Likewise, it follows from Lemma \ref{lem-N-struc} that there are $j\in \{1,2,\ldots, n\}$ with $\e(e_{j}')=-$ and $l\geq 1$ such that 
\begin{align}
\mathsf{u}'= \mathsf{s}' (e_{j+1}'|\ol{e_j'})^{-1}\cdots (\ol{e_{j+l-1}'}|e_{j+l}') \mathsf{t}', \notag
\end{align}
where $\mathsf{s}'=1_{\{e_j',\ol{e_j'}\}}$ or $(f'|e_j')$ with cyclic subordering $(e_{j}', \ol{e_{j-1}'}, f')_{s(e_{j}')}$, and $\mathsf{t}'=1_{\{e_{j+l}',\ol{e_{j+l}'}\}}$ or $(g'|\ol{e_{j+l}'})^{-1}$ with cyclic subordering $(\ol{e_{j+l}'}, e_{j+l+1}', g')_{s(\ol{e_{j+l}'})}$.

Without loss of generality, we can assume $\mathsf{u}=\mathsf{u}'$; otherwise, we replace $\mathsf{u}$ by the string associated to $\ol{w}$.
Combining the two conditions deduced from Lemma \ref{lem-M-struc} and Lemma \ref{lem-N-struc} above, we get that 
\begin{align}
\begin{array}{rlcl}
\text{either }& \mathsf{s}=1_{\{ e_{i}, \ol{e_{i}}\}}, & \text{which is equivalent to }& \mathsf{s}'=(f'|e_j')=(\ol{e_i}|e_{i+1}), \\
\text{or }& \mathsf{s}=(e_i|f)^{-1}=(e_{j+1}'|\ol{e_j'})^{-1}, & \text{which is equivalent to }& \mathsf{s}'=1_{\{e_j',\ol{e_j'}\}}. \\
\end{array}\notag 
\end{align}
Therefore, if $\mathsf{s}=1_{\{e_i,\ol{e_i}\}}$, then we have $f'=\ol{e_{i}}$, $t_{1}:=e_{i+1}=e_{j}'$, and cyclic subordering $(\ol{e_{j-1}'}, \ol{e_{i}}, t_{1})_{s(t_{1})}$.
Note that this is the cyclic subordering $(\ol{e_{j-1}'}, \ol{e_{i-1}}, t_{1})_{s(t_{1})}$ stated in (L2) if we shift the index $i$.

On the other hand, if $\mathsf{s}=(e_i|f)^{-1}$, then we have $f=\ol{e_{j}'}$, $t_{1}:=e_{i}=e_{j+1}'$ and cyclic subordering $(\ol{e_{j}'}, \ol{e_{i-1}}, t_{1})_{s(t_{1})}$.
Shifting the index $j$ yields the required cyclic subordering around $s(t_1)$. 

Similarly, the equivalences
\begin{align}
\begin{array}{rlcl}
& \mathsf{t}=1_{\{e_{i+k},\ol{e_{i+k}}\}} & \Leftrightarrow& \mathsf{t}'=(g'|\ol{e_{j+l}'})^{-1} = (e_{i+k}|\ol{e_{i+k-1}})^{-1}, \\
\text{and }& \mathsf{t}=(\ol{e_{i+k}}|g)=(\ol{e_{j+l-1}'}|e_{j+l}') & \Leftrightarrow& \mathsf{t}'=1_{\{e_{j+l}',\ol{e_{j+l}'}\}},
\end{array} \notag 
\end{align}
yield the required cyclic subordering $(\ol{t_{\ell}}, e_{j+\ell}, e_{i+\ell})_{s(\ol{t_{\ell}})}$ for some $\ell \ge 1$.
Since $t_{2}, \ldots, t_{\ell-1}$ are uniquely determined by $t_{1}$ and $t_{\ell}$, we have the desired signed half-walk $z=(t_{1}, t_{2}, \ldots, t_{\ell})$.

Conversely, given (L2) we can construct $\mathsf{u}, \mathsf{u}'$ in the form described above.
Hence, the claim follows from Lemma \ref{lem-M-struc} and Lemma \ref{lem-N-struc}.
\end{proof}

\begin{example}\label{eg-loewy2}
Consider again the Brauer graph and the walk $w_1$ used in Example \ref{eg-loewy0}, \ref{eg-loewy1}.
There is a maximal common subwalk given by $z_2=(4^-)\in w_1\cap \ol{w_1}$ which looks locally in $\G$ as follows (cf. figures in Example \ref{eg-walk1}):
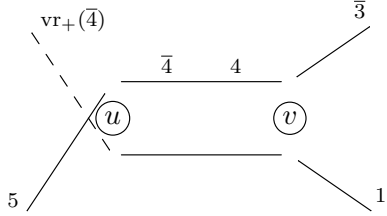
\[
\xymatrix@R=5pt{
\ar@{--}[rdddd]^<{\vir_+(\ol{4})}& & & &\ar@{-}[ldd]_<{\ol{3}}\\
& & & &\\
& \ar@{-}[rr]^{\ol{4} \qquad 4} & & &\\
& *+[Fo]{u} & & *+[Fo]{v}& \\
& \ar@{-}[rr] & & &\\
& & & & \\
\ar@{-}[ruuuu]^<{5}& & & & \ar@{-}[luu]_<{1}
}
\]
Here we use a dashed line to represent the virtual edge $\vir_+(\overline{4})$ as in Example \ref{eg-walk1}.
Taking $w=\ol{w_1}$, $w'=w_1$, and $z=z_2$, we see that the condition (L2) holds.
The common substring of $\word,\word'$ which corresponds to this crossing is $(5|\ol{4})(\ol{3}|4)^{-1}$.
For clarity, we underlined the appearance of this substring in $\word,\word'$ as follows:
\begin{align}
\word&= (1|\ol{6})^{-1}(\ol{1}|2)(3|\ol{2})^{-1}(\ol{3}|\ol{6})(\ol{5}|6)^{-1}\underline{(5|\ol{4})(\ol{3}|4)^{-1}}(3|\ol{2})(\ol{1}|2)^{-1}(1|\ol{3}), \notag \\
\word'&= (3|\ol{2})^{-1}(\ol{3}|\ol{6})(\ol{5}|6)^{-1}(5|\ol{4})(\ol{3}|4)^{-1}(3|\ol{2})(\ol{1}|2)^{-1}(1|\ol{3})\underline{(\ol{3}|4)(5|\ol{4})^{-1}}. \notag 
\end{align}

For the other maximal subwalk given by $z_1=(1^+,2^-,3^+)\in w_1\cap\ol{w_1}$, we can visualise the local structure around the endpoints as follows:
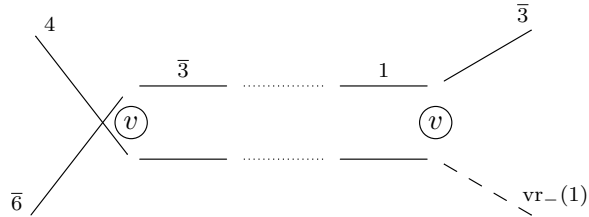
\[
\xymatrix@R=5pt@C=30pt{
\ar@{-}[rdddd]^<{4}& & & & &\ar@{-}[ldd]_<{\ol{3}}\\
& & & & &\\
& \ar@{-}[r]^{\ol{3}} & \ar@{.}[r]& & \ar@{-}[l]_{1} &\\
& *+[Fo]{v} & & & *+[Fo]{v}& \\
& \ar@{-}[r] & \ar@{.}[r]& &\ar@{-}[l]&\\
& & & & \\
\ar@{-}[ruuuu]^<{\ol{6}}& & & & & \ar@{--}[luu]_<{\vir_-(1)}
}
\]
This indicates that $z_1$ satisfies the condition (L2).
The corresponding common substring of $\word,\word'$ is given by $(1|4)^{-1}(\ol{1}|2)(3|\ol{2})^{-1}(\ol{3}|4)$, which is underlined as follows:
\begin{align}
\word&= (4|\ol{6})^{-1}\underline{(1|4)^{-1}(\ol{1}|2)(3|\ol{2})^{-1}(\ol{3}|4)}(4|\ol{6})(\ol{5}|6)^{-1}(5|\ol{4})(\ol{3}|4)^{-1}(3|\ol{2})(\ol{1}|2)^{-1}(1|\ol{3}), \notag \\
\word'&= (3|\ol{2})^{-1}(\ol{3}|\ol{6})(\ol{5}|6)^{-1}(5|\ol{4})\underline{(\ol{3}|4)^{-1}(3|\ol{2})(\ol{1}|2)^{-1}(1|4)}(5|\ol{4})^{-1}. \notag 
\end{align}
\end{example}

The following proposition is the final piece needed to prove Theorem \ref{mainbij}.
\begin{proposition}\label{lem-hom-to-comb}
\begin{itemize}
\item[(1)] Retaining the notations used so far, $\Hom_{\Kb(\proj\Lambda_\G)}(T',T[1])=0$ if and only if neither (U2) nor (L2) holds for $W=W_{T}$ and $W'=W_{T'}$.
\item[(2)] For two (not necessarily distinct) complexes $U,V\in \gcx\Lambda_\G$, the hom-spaces $\Hom_{\Kb(\proj\Lambda_\G)}(U,V[1])$ and $\Hom_{\Kb(\proj\Lambda_\G)}(V,U[1])$ are simultaneously zero if, and only if $W_{U},W_{V}$ are non-crossing and satisfy the sign condition.
\end{itemize}
\end{proposition}
Note that the right-hand side of (2) is not equivalent to saying $\{W_U,W_V\}$ is admissible, as it does not require $W_U$ and $W_V$ to be admissible.
\begin{proof}
(1) This follows from combining Lemma \ref{lem-non-zero-hom}, \ref{lem-L-uniserial}, and \ref{lem-L-caseC}.

(2) First, we use the following table to clarify the relation between the non-crossing and signs conditions, and the conditions in (U2), (L2).
In each row of the table, the condition in the first entry fails to hold for $\{W_U, W_V\}$ implies that one of the conditions in the second entry holds for $(W,W')=(W_U,W_V)$ or $(W,W')=(W_V, W_U)$.
For example, if $\{W_U,W_V\}$ fails (NC1), then at least one of the pairs $(W_U,W_V)$ or $(W_V,W_U)$ satisfies one of the conditions (U2)(i) or (U2)(ii).
On the other hand, if a condition appears on the right-hand side of the table, then one of the corresponding left-hand side conditions will fail.
For example, say (U2)(i) holds for $(W_U,W_V)$, then $\{W_U,W_V\}$ fails (NC1) or the sign condition.
\begin{center}
\begin{tabular}{c|ll}
$\begin{array}{c} \text{Condition which fails}\\ \text{for }\{W_U,W_V\} \end{array}$  &  &  Condition which holds for $(W_U,W_V)$ or $(W_V,W_U)$ \\
\hline 
(NC1) & & (U2)(i) or (ii) \\
(NC2) & & (L2) \\
(NC3) & & (U2)(iii) or (iv) \\
sign condition & & (U2)(i) or (iii)
\end{tabular}
\end{center}
Now the statement follows by applying (1) to $(T,T')=(U,V)$ and to $(T,T')=(V,U)$ simultaneously.
\end{proof}

\begin{proof}[Proof of Theorem \ref{mainbij}]
Suppose $U,V$ are (possibly the same) complexes in $\gcx\Lambda_\G$.
By definition, $U\oplus V$ being pretilting means that the four hom-spaces $\Hom(U,U[1])$, $\Hom(U,V[1])$, $\Hom(V,V[1])$, $\Hom(V,U[1])$ are all zero. By Proposition \ref{lem-hom-to-comb}, this is equivalent to $\{W_U,W_V\}$ being an admissible set.
This implies that a (basic) two-term pretilting complex corresponds to precisely an admissible set of signed walks.  The first bijection in the statement of the theorem is just the special case when $U=V$.
More generally, we get that $T_\W$ is pretilting for any admissible set $\W$ of signed walks.

Since any two-term pretilting complex $T$ is partial, if $T\oplus U$ is a two-term pretilting complex with $T$ tilting, then $U\in \add T$.  
Translating this to the combinatorial side, we obtain that $T_\W$ is tilting precisely when $\W$ is complete.
\end{proof}

\begin{proof}[Proof of Proposition \ref{2tilt}]
Note that the set $\AW(\G)$, and hence $\CW(\G)$, does not depend on the multiplicity of $\G$; i.e. we have natural bijections $\AW(\G)\to\AW(\G')$ and $\CW(\G)\to\CW(\G')$.
These yield bijections $\psi:\ipt\Lambda_\G\to\ipt\Lambda_{\G'}$ and
(by abusing notation) $\psi:\ttilt\Lambda_\G\to\ttilt\Lambda_{\G'}$ by Theorem \ref{mainbij}.

We are left to show that $\psi$ preserves the partial order.
Let $T,T'\in\ttilt\Lambda_\G$ with $T\geq T'$.
This means that $\Hom(X,X'[1])=0$ for all indecomposable summands $X$ and $X'$ of $T$ and $T'$ respectively.
Let $W_X$ be the signed walk corresponding to an indecomposable summand $X$ of $T$.
Recall from the construction (cf. Lemma \ref{lem-SW-biject}) that $W_X$ is independent of the multiplicities on the vertices.
Hence, the indecomposable summand $\psi(X)$  of $\psi(T)$ (resp. $\psi(X')$ of $\psi(T')$) corresponds also to $W_X$ (resp. $W_{X'}$).
Therefore, as Proposition \ref{lem-hom-to-comb} (1) is independent of multiplicity, we have $\Hom(\psi(X), \psi(X')[1])=0$; hence, $\psi(T)\geq \psi(T')$, i.e. $\psi$ is order-preserving.
\end{proof}

\section{Tilting-discrete Brauer graph algebras}\label{sec-application}
\subsection{Preliminaries}\label{sec-application-prelim}
We first recall some results about tilting mutation theory from \cite{AI,A,AIR}. 
Throughout this subsection, $\Lambda$ is assumed to be a finite dimensional \emph{symmetric} algebra.  Most of the facts stated here have analogues for general finite dimensional algebras by replacing the word ``tilting" with ``silting"; see \emph{loc. cit.} for the details.

Let $\mathcal{A}$ be a full additive subcategory of $\CC=\Kb(\proj\Lambda)$ or $\mod\Lambda$.
A map $f:X\to Y$ in $\CC$ is \emph{left minimal} if all maps $g:Y\to Y$ with $gf=f$ are isomorphisms.
A map $f:U\to X$ in $\CC$ is called a \emph{left $\mathcal{A}$-approximation} of $U$
if $X$ belongs to $\mathcal{A}$ and $\Hom_{\CC}(f,C)$ is surjective for any $C$ in $\mathcal{A}$.
We say that $\mathcal{A}$ is \emph{covariantly finite} in $\CC$ if 
for all $U$ in $\CC$, there exists a left $\mathcal{A}$-approximation.
Dually, we define right minimality, right $\mathcal{A}$-approximations, and contravariantly finite subcategories in $\CC$.
A full subcategory is called \emph{functorially finite} if it is covariantly and contravariantly finite in $\CC$. 

\begin{definitiontheorem}\cite{AI}
Let $T$ be a basic tilting complex in $\Kb(\proj\Lambda)$ with a decomposition $T=X\oplus M$.
A \emph{left tilting mutation} $\mu_X^-(T)$\label{sym-muXT} of $T$ with respect to $X$ is a (basic) tilting complex given by $Y\oplus M$, where $Y$ is the (well-defined) object fitting into the following triangle:
\[\xymatrix{
X \ar[r]^f & M' \ar[r] & Y \ar[r] & X[1]
}\]
where $f$ is a minimal left $\add M$-approximation of $X$.
Dually, one also has another tilting complex given by \emph{right tilting mutation} $\mu_X^+(T)$ of $T$ with respect to $X$.

A \emph{tilting mutation} means a left or right tilting mutation; it is called \emph{irreducible} if $X$ is indecomposable.
\end{definitiontheorem}

\begin{definition}\cite{AI}
The \emph{tilting quiver} of $\Lambda$, denoted by $\T_\Lambda$\label{sym-TLambda}, is defined as follows:
\begin{itemize}
\item The set of vertices is $\tilt\Lambda$.
\item Draw an arrow $T\to U$ if $U$ is an irreducible left mutation of $T$.
\end{itemize}  Note that this is precisely the Hasse quiver of the poset $(\tilt\Lambda,\geq)$.
A connected component of $\T_\Lambda$ is said to be \emph{canonical} if it contains $\Lambda$.
\end{definition}

\begin{definition}
A symmetric algebra $\Lambda$ is said to be \emph{tilting-connected} if
the tilting quiver of $\Lambda$ is connected.
We say that $\Lambda$ is \emph{tilting-discrete} if for any positive integer $\ell$,
there exist only finitely many tilting complexes $T$ in $\tilt\Lambda$ satisfying $\Lambda\geq T\geq\Lambda[\ell]$.
\end{definition}
It was shown in \cite[Section 3]{A} that if $\Lambda$ is tilting-discrete, then it is tilting-connected.

Two classes of tilting-discrete symmetric algebras are 
found in the second author's previous works.  Namely, the local  algebras in \cite{AI} and the representation-finite algebras in \cite{A}.

We refer to \cite{AI} for more general examples of tilting/silting-connected algebras.
On the other hand, in a joint work of the second author with Grant and Iyama \cite{AGI}, we know that there exist non-tilting-connected symmetric algebras.
In any case, it is not easy to answer the following question.

\begin{question}
When is a symmetric algebra tilting-connected, or even tilting-discrete?
\end{question}

The next subsection is devoted to answering the tilting-discrete part of the question for Brauer graph algebras.

We first look at some properties when the set $\ttilt\Lambda$ is finite, or equivalently (by Proposition \ref{prop-noCommonSummand} (ii)), when the set $\ipt\Lambda$ is finite.

\begin{proposition}\cite{A}\label{prop-locTiltDiscreteMutate}
If $\ttilt\Lambda$ is a finite set,
then any two-term tilting complex can be obtained from $\Lambda$ by iterated irreducible left tilting mutation.
\end{proposition}

The following is a special case of \cite[Theorem 2.4]{AM}.

\begin{proposition}\cite{AM}
\label{tilting-discrete}
Let $\Lambda$ be a symmetric algebra.
If for any tilting complex $P$ in the canonical connected component of $\T_\Lambda$, the set 
$\ttilt\End_{\Kb(\proj\Lambda)}(P)$ is finite,
then $\Lambda$ is tilting-discrete.
In particular, it is tilting-connected.
\end{proposition}

\subsection{Conditions for tilting-discreteness}\label{sec-tilt-discrete}
The aim of this subsection is to show the following theorem.

\begin{theorem}\label{TD of type odd}
Let $\G$ be a Brauer graph.
Then the following are equivalent:
\begin{enumerate}[{\rm (i)}]
\item $\Lambda_\G$ is tilting-discrete;
\item $\ttilt\Lambda_\G$ is a finite set;
\item $\G$ contains at most one odd cycle and no even cycle.
\end{enumerate}
\end{theorem}

First of all, we start by recalling the Brauer graph mutation theory from \cite{Kau,A2}.
\begin{definition}[Flips of Brauer graph] \label{def-flip}
Let $\G=(V,H,s,\ol{\,\cdot\,},\sigma;\mm)$ be a Brauer graph and $E$ be an edge of $\G$.  If $\G$ has more than one edge, then the \emph{left flip of $\G$ at $E$} is a Brauer graph $\mu^-_E(\G)$\label{sym-muEG} given by $(V,H,s',\ol{\,\cdot\,},\sigma';\mm)$, where $s'$ and $\sigma'$ are defined as follows.
\[
\begin{array}{c|l|c|c}
\text{half-edge $h$} & \multicolumn{1}{|c|}{\text{condition}} & s'(h) & \sigma'(h) \\ [0.2em] \hline
\hline
e \in E & \sigma^{-1}(e)\notin E;\;\; \sigma(e)\neq \ol{e}    & s\left(\ol{\sigma^{-1}(e)}\right) & \ol{\sigma^{-1}(e)} \rule{0pt}{1.5em} \\ [0.6em]
 & \sigma^{-1}(e)\notin E;\;\; \sigma(e)=\ol{e} & s\left(\ol{\sigma^{-1}(e)}\right) & \ol{e}   \\[0.6em]
 & \sigma^{-1}(e) = e     & s(e)   & e        \\[0.6em]
 & \sigma^{-1}(e) =\ol{e} & s\left(\ol{\sigma^{-1}(\ol{e})}\right)=s'(\ol{e})   & \ol{\sigma^{-1}(\ol{e})}                     \\[0.6em]
\hline
f\notin E & f=\sigma^{-1}\left(\ol{\sigma^{-1}(e)}\right)& s(f) & e \rule{0pt}{1.5em} \\ [0.6em]
          & f=\sigma^{-1}(e);\;\; \sigma(e)=\ol{f}        & s(f) & e \\ [0.6em]
          & f=\sigma^{-1}(e);\;\; \sigma(e)=\ol{e};\;\;\; \sigma(\ol{e})=\ol{f}     & s(f) & e \\ [0.6em]
          & f=\sigma^{-1}(e);\;\; \sigma(e)=\ol{e};\;\;\; \sigma(\ol{e})\neq\ol{f}  & s(f) & \sigma^2(e) \\ [0.6em]
          & f=\sigma^{-1}(e);\;\; \sigma(e)\neq\ol{e},\ol{f}     & s(f) & \sigma(e)  \\ [0.6em]
          & \text{all other cases}          & s(f) & \sigma(f) \\ [0.6em]
\end{array}
\]
If $\G$ has one edge, then the left flip of $\G$ is defined to be itself.

The \emph{opposite (Brauer) graph} of $\G$ is the Brauer graph $\G^\op = (V,H,s,\ol{\,\cdot\,},\sigma^{-1})$\label{sym-Gop}.  The \emph{right flip of $\G$ at $E$} is the Brauer graph $\mu_E^+(\G):=\mu_E^-(\G^\op)^\op$.
\end{definition}

The simplest way to present a graphical presentation of the left flip is given below.
If $E$ is not a loop, i.e. $u,v$ in the graphs below are distinct vertices:
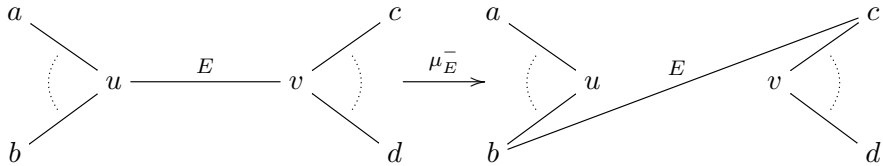
\[
\xymatrix@C=25pt@R=15pt{
a \ar@{-}[dr]_{}="au"& & & & c \ar@{-}[dl]^{}="cu" & a \ar@{-}[dr]_{}="au2"& & & & c \ar@{-}[dl]^{}="cu2"\\
  & u \ar@{-}[rr]^{E} & &  v&  \ar[r]^{\mu_E^-}&   & u  & & v&  \\
b \ar@{-}[ur]^{}="bu"& & & & d \ar@{-}[ul]_{}="du" & b \ar@{-}[uurrrr]^{E}\ar@{-}[ur]^{}="bu2"& & & & d \ar@{-}[ul]_{}="du2"
\ar@{.}@/_/ "au";"bu"
\ar@{.}@/^/ "cu";"du"
\ar@{.}@/_/ "au2";"bu2"
\ar@{.}@/^/ "cu2";"du2"
}
\]
Here we allow some (or all) of $a,b,v$ to be the same vertex; similarly for $c,d,u$.

If $E$ is a loop, then we have one of the following two cases.
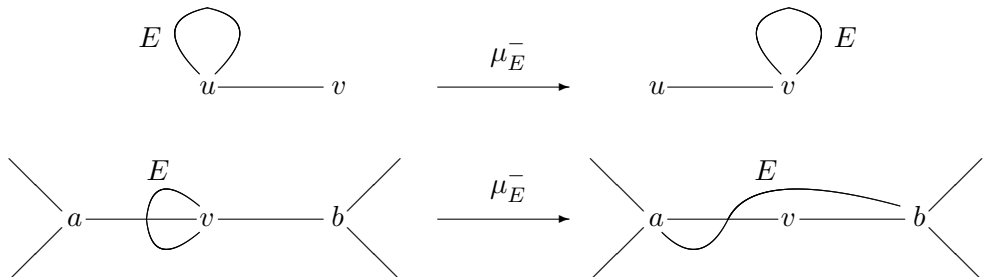
\begin{center}
\begin{picture}(250,120)(0,-10)
\put(0,7.5){{$v$}}
\put(-50,7.5){{$a$}}
\put(50,7.5){{$b$}}
\put(-43,10){\line(1,0){42}}
\put(7,10){\line(1,0){40}}
\put(-52,13){\line(-1,1){20}}
\put(-52,7){\line(-1,-1){20}}
\put(56,13){\line(1,1){20}}
\put(56,7){\line(1,-1){20}}
\qbezier(0,15)(-18,30)(-20,10)
\qbezier(-20,10)(-18,-10)(0,5)
\put(-20,25){{$E$}}

\put(90,10){\vector(1,0){50}}
\put(110,20){{$\mu_E^-$}}

\put(220,7.5){{$v$}}
\put(170,7.5){{$a$}}
\put(270,7.5){{$b$}}
\put(177,10){\line(1,0){42}}
\put(227,10){\line(1,0){40}}
\put(168,13){\line(-1,1){20}}
\put(168,7){\line(-1,-1){20}}
\put(276,13){\line(1,1){20}}
\put(276,7){\line(1,-1){20}}
\qbezier(265,15)(210,30)(200,10)
\qbezier(200,10)(190,-10)(175,5)
\put(210,25){{$E$}}

\put(0,57.5){{$u$}}
\put(50,57.5){{$v$}}
\put(7,60){\line(1,0){40}}
\qbezier(0,65)(-20,85)(3,90)
\qbezier(3,90)(26,85)(6,65)
\put(-23,75){{$E$}}

\put(90,60){\vector(1,0){50}}
\put(110,70){{$\mu_E^-$}}

\put(170,57.5){{$u$}}
\put(220,57.5){{$v$}}
\put(177,60){\line(1,0){40}}
\qbezier(220,65)(200,85)(223,90)
\qbezier(223,90)(246,85)(226,65)
\put(240,75){{$E$}}
\end{picture}
\end{center}
Here $a,b$ can be the same vertex. cf. \cite[Section 5]{A2}
\begin{remark}
\begin{enumerate}[(1)]
\item We will abuse the notation $\mu_E^\pm(\G)$ to mean ``$\mu_E^-(\G)$, and respectively $\mu_E^+(\G)$," for ease of stating results.  Similar abuses of notations will also be adopted for mutations of tilting complexes.
\item Intuitively, if one draws $\G$ on a piece of paper and places a mirror perpendicular to the paper next to $\G$, then $\G^\op$ is the reflection of $\G$ in the mirror.
\item The opposite ring $(\Lambda_\G)^\op$ of $\Lambda_\G$ is isomorphic to $\Lambda_{\G^\op}$.
\item One can also define the right flip explicitly, then the left flip will be given by $\mu_E^-(\G)=\mu_E^+(\G^\op)^\op$.
\item The left/right flip was found by Kauer, \cite[Lemma 2.7-2.10]{Kau2} details all possible situations.
This operation is termed as Kauer move in \cite{MS}; the terminology ``flip" was adopted in \cite{A} to align with the flip of triangulations in cluster mutation theory. 
Indeed, if the underlying ribbon graph of $\G$ is a triangulation of a Riemann surface, then $\mu_E^\pm(\G)$ is precisely the flip of the triangulation at the arc $E$.
\end{enumerate}
\end{remark}

\begin{proposition}{\rm \cite{Kau2}}\label{prop-flip der equiv}
Let $\G$ be a Brauer graph and $E$ be an edge of $\G$.  Then the endomorphism ring of the tilting complex $\mu_{P_E}^\pm(\Lambda_\G)$ is isomorphic to $\Lambda_{\mu_E^\pm(\G)}$.
\end{proposition}
\begin{remark}\label{rmk-der-eq}
An English translation for part of the proof by Kauer can be found in \cite{Kau}.
For the ease of reference, we attach the complete proof in the Appendix \ref{sec-app-endo}.
Another explicit calculation can be found in \cite{Dem}, which proves the analogue of this proposition for a more generalised class of algebras.
Other proofs were attempted in \cite{An2,A2}, but these proofs rely on a claim from \cite{Pog}, namely, that any algebra stably equivalent to a self-injective special biserial algebra is also special biserial.
However, the proof of this claim is incorrect \cite[Appendix]{AIP}; the validity of the claim is still unknown at the time of writing.
In the case when the underlying field is not of characteristic 2, we were informed by Alexandra Zvonareva that these proofs can be fixed by \cite{AZ2}, where they prove that the class of Brauer graph algebras is closed under derived equivalences.
\end{remark}

\begin{lemma}\label{flip and type odd}
Let $\G$ be a Brauer graph and $E$ an edge of $\G$.
If $\G$ contains $c$ odd cycles with $c\in\{0,1\}$, and no even cycle, then so does $\mu_E^\pm(\G)$.
\end{lemma}
\begin{proof}
First note that $\mu_E^\pm(\G)$ preserves connectedness.  In particular, if $\G$ is a (connected) tree (i.e. $c=0$), or equivalently $|V|=|H/\ol{\,\cdot\,}|+1$, then so is $\mu_E^\pm(\G)$.

Assume now that $\G$ has exactly one odd cycle and has no even cycle.  Now we have $|V|=|H/\ol{\,\cdot\,}|$, and by the same argument, $\mu_E^\pm(\G)$ must then contain precisely one cycle.  We are left to show that the parity of the cycle length remains unchanged after a flip.

Since flipping an edge does not alter the rest of the graph, the odd cycle of $\G$ stays as the same subgraph if we flip at an edge which is not in the cycle.  If $E$ is contained in the odd cycle, say of length $\ell$, of $\G$, then observe using the graphical presentation of flips that the cycle length can only be $\ell$, or $\ell-2$, or $\ell+2$.
\end{proof}

Now we are ready to prove Theorem \ref{TD of type odd}.
\begin{proof}[Proof of Theorem \ref{TD of type odd}]
It is evident that the implication (i)$\Rightarrow$(ii) holds.

The implications (ii)$\Leftrightarrow$(iii) follows from Theorem \ref{mainbij} and Proposition \ref{type-odd}. 

We show that the implication (iii)$\Rightarrow$(i) holds.
Let $P$ be a tilting complex of $\Lambda_\G$ in the canonical connected component of $\T_{\Lambda_\G}$.
Then it follows from Proposition \ref{prop-flip der equiv} that the endomorphism algebra $\Gamma:=\End_{\Kb(\proj\Lambda_\G)}(P)$ of $P$ is a Brauer graph algebra
whose Brauer graph $\G'$ is obtained by a series of left and right flips starting from $\G$.
By Lemma \ref{flip and type odd}, $\G'$ then has the same number (zero or one) of odd cycles as $\G$, and it also has no even cycle.  Combining Theorem \ref{mainbij} and Proposition \ref{type-odd}, we get that $\ttilt\Gamma$ is a finite set.  Hence, $\Lambda_\G$ is tilting-discrete by Proposition \ref{tilting-discrete}.
\end{proof}

The following corollary is immediate from (the proof of) Theorem \ref{TD of type odd}.

\begin{corollary}\label{cor-type-odd-equiv}
Let $\G$ be a Brauer graph which contains at most one odd cycle and no even cycle, and let $T$ be a tilting complex of $\Lambda_\G$.
Then the endomorphism algebra of $T$ in $\Kb(\proj\Lambda_\G)$ is isomorphic to $\Lambda_{\G'}$ for some $\G'$ with the same number of odd cycles, no even cycle, and the same multiplicity as $\G$.  In particular, any algebra derived equivalent to $\Lambda_\G$ is also a Brauer graph algebra.
\end{corollary}

The class of algebras described in Theorem \ref{TD of type odd} appears in \cite{An1} as (precisely) the class of Brauer graph algebras with non-degenerate Cartan matrices.  Note that the Grothendieck group of the stable module category of such a Brauer graph algebra is finite, and vice versa.

For an arbitrary symmetric algebra $\Lambda$, we do not know if the non-degeneracy of its Cartan matrix, or the finiteness of the set $\ttilt\Lambda$, are equivalent conditions for tilting-discreteness.  We also do not know if the finiteness, or even if the number of elements, of the set $\ttilt\Lambda$ is derived invariant.  In fact, one of the original motivations of this work and its sequel is to see if one can take advantages of the rich combinatorics of Brauer trees in order to count the number of two-term tilting complexes for Brauer tree algebras.

Suppose $\G$ is a Brauer graph such that $\Lambda_\G$ is tilting-discrete. 
Let $\G_0$ be the associated \emph{multiplicity-free Brauer graph}, i.e. the Brauer graph with the same ribbon graph structure but multiplicity $\mm\equiv 1$.
Then the associated Brauer graph algebra $\Lambda_{\G_0}$ is of finite type when $c=0$, or of one-parametric Euclidean type when $c=1$.  See for example \cite{Sko} for the details.

Also note that the multiplicity-free Brauer graph algebras form precisely the class of trivial extensions of gentle algebras \cite{Sch}.  For readers who are familiar with silting theory \cite{AI} and silting-discreteness \cite{A}, \cite[Proposition 6.8]{BPP} asserts that all derived-discrete algebras of finite global dimension are silting-discrete.  However, the trivial extensions of these algebras also contain multiplicity-free Brauer graph algebras which lie outside the class presented in Theorem \ref{TD of type odd}.  In particular, this shows that silting-discreteness is one of the many properties destroyed by taking trivial extension.

Although Theorem \ref{TD of type odd} gives the precise condition for a Brauer graph algebra to be tilting-discrete, 
we still do not know if there is a tilting-connected non-tilting-discrete Brauer graph algebra.  
We remark that, as shown in another on-going work of the second named author with Grant and Iyama \cite{AGI}, 
the Brauer graph algebra whose underlying graph is a digon (i.e. cycle of length 2) is neither tilting-discrete nor tilting-connected.

Our final remark to this result is the problem of classifying derived equivalence classes of Brauer graph algebras.
Although Kauer has shown in \cite{Kau,Kau2} that any Brauer graph algebra is derived equivalent to a special type of Brauer graph (called Brauer double star) algebra, it is unknown in general whether an algebra derived equivalent to a Brauer graph algebra must also be a Brauer graph algebra.
At the time of writing, it seems that this problem can be solved at least in the case when the characteristic of the underlying field is not 2 \cite{AZ2} - as we have already mentioned in Remark \ref{rmk-der-eq}.
Moreover, there exist distinct Brauer double stars whose associated algebras are derived equivalent in general.
There are some special cases (for instance, when the underlying graph is a tree, c.f. \cite{MH}) where this choice is unique up to a rearrangement of the multiplicities.
We also remark that \textit{loc. cit.} only shows Kauer's result in a subclass (called generalised Brauer trees), instead of determining the derived equivalence class for this subclass.

As far as we know, three derived equivalence classes are already known before our work; more detailed information can be found in the survey \cite{Sko}.
The first one is the class of all Brauer tree algebras, i.e. the Brauer graph is a tree, with at most one vertex having multiplicity greater than 1.
The second class is given by Brauer graphs which have at most one odd cycle, no even cycle, and all multiplicities being 1.
The third class is given by Brauer graphs whose underlying graphs are trees, all but two vertices have multiplicity 1, and the two exceptional vertices have multiplicities (at most) 2.
Therefore, Corollary \ref{cor-type-odd-equiv} is a generalisation of these results but with an entirely different, and multiplicity-independent, approach.

\pagebreak

\appendix

\section{On the endomorphism algebra of an irreducible mutation}\label{sec-app-endo}

The aim of this appendix is to show:
\begin{lemma}\label{lem-flip-alg}
For a Brauer graph $\G$, let $\Gamma =\Bbbk Q/I$ be the endomorphism algebra of the tilting complex $\mu_E^-(\Lambda_\G)$, where $E=\{e,\ol{e}\}$ is an edge of $\G$.
Then $\Gamma$ is symmetric special biserial.
\end{lemma}

The original proof in \cite{Kau2} were split into multiple lemmas, and only shown explicitly in the easier cases.
Moreover, the cited reference, which is the thesis of Kauer written completely in German, is not accessible to the general public.
Thus, we attach hereby an explicit computation which handle (almost) all cases for ease of reference.
Our notation here is also specifically chosen so that one can see immediately that $\End(\mu_E^-(\Lambda_\G)) \cong \Lambda_{\mu_E^-(\G)}$.

\medskip

As we have mentioned in Remark \ref{rmk-der-eq}, a step (\cite[Proposition 2.5]{A2}) in the proof of Proposition \ref{prop-flip der equiv} relies on an unproved claim in \cite{Pog}.
The only use of this step in \cite{A2} is to show that $\Gamma$ is symmetric special biserial - specifically, the first step of \cite[6.3.1 (2)]{A2}, and also that of \cite[6.3.2 (2)]{A2}.

Lemma \ref{lem-flip-alg} fills the gap in the proof of the main result of \cite{A2}, i.e. Proposition \ref{prop-flip der equiv}.
Note that the same statement for the endomorphism algebra of the right mutation $\mu_{P_E}^+(\Lambda_\G)$ can be proved dually, and we will not give any detail here.
We start our proof now.

Since $\Gamma$ is derived equivalent to $\Lambda_\G$, it follows from a result of Rickard that $\Gamma$ is symmetric.

By definition, the algebra $\Gamma$ is special biserial if the following are satisfied:
\begin{itemize}
\item[(SB1)] Any vertex of $Q$ is the head of at most two arrows.
\item[(SB1')] Any vertex of $Q$ is the tail of at most two arrows.
\item[(SB2)] For an arrow $\beta\in Q_1$, there is at most one arrow $\alpha$ with $\alpha\beta \notin I$.
\item[(SB2')] For an arrow $\beta\in Q_1$, there is at most one arrow $\gamma$ with $\beta\gamma \notin I$. 
\end{itemize}

(SB1) and (SB1') follows from \cite[6.3.1(1),6.3.2(1)]{A2}, which shows that the quiver $Q$ of $\Gamma$ is the same as $Q_{\mu_E^-(\G)}$, so it remains to check that (SB2) and (SB2') hold.
For covenience, we say that (SB2) holds for an arrow $\beta$, if there is at most one arrow $\alpha$ with $\alpha\beta\notin I$; similarly for (SB2').

We can use Definition \ref{def-flip} to see that $Q$ changes at most six of the arrows from $Q_\G$.
Recall that we identify arrows in $Q_\G$ with irreducible maps between indecomposable projective modules.
Similarly, we can identify arrows in $Q$ with irreducible maps between indecomposable pretilting complexes.
Note that the vertex $E\in (Q_\G)_0=Q_0$ now represents a two-term pretilting complex $T$ in place of $P_E$.
In what follows, we present explicitly those maps (arrows) that are removed from $Q_\G$ and the new maps that are added to form $Q$.
For convenience, define $f$ (resp. $f'$) to be $\sigma^{-1}(e)$ (resp. $\sigma^{-1}(\ol{e})$) whenever $e$ (resp. $\ol{e}$) is not truncated, and $g$ (resp. $g'$) to be $\sigma^{-1}(\ol{f})$ (resp. $\sigma^{-1}(\ol{f'})$) whenever $\ol{f}$ (resp. $\ol{f'}$) is not truncated.

There are four cases to consider.

\subsection*{Case 1}
Suppose $E$ is a loop such that there is some $e\in E$ with $\ol{e}=\sigma(e)$.
The summand $P_E$ in $\Lambda_\G$ is replaced by  $T:=(P_E \xto{d} P_{f} \oplus P_{f})$, where $d=[(f|e), (f|\ol{e})]^t$, to form $\mu_E^-(\Lambda_\G)$.

\subsection*{Case 1a}
Suppose furthermore that $\ol{f}\neq \sigma(\ol{e})$, or equivalently $g\neq\ol{e}$.

If $\ol{f}$ is not truncated in $\G$ and $f$ is not truncated in $\mu_E^-(\G)$, then the quiver $Q$ is given by removing the subset $\{\alpha_{\ol{e},\sigma(\ol{e})}, \alpha_{e,\ol{e}}, \alpha_{f,e}, \alpha_{g,\ol{f}} \} $
of $Q_\G$ and adding the arrows given by
\begin{align}
\beta_{f,\sigma(\ol{e})} :=  \alpha_{f,\sigma(\ol{e})}:P_{\sigma(\ol{e})} \to P_{f}  \hspace*{3cm}\notag 
\end{align}
\begin{align}
\xymatrix@R=35pt{ T \;\;=\ar@<-9pt>[d]_{\beta_{g,e}}  \\ P_g\;\;=} \quad 
\xymatrix@C=70pt@R=35pt{ (P_E \ar[r]|{\left[\begin{smallmatrix}(f|e)\\(f|\ol{e})\end{smallmatrix}\right]} \ar[d]_{0} & P_f\oplus P_f)\ar[d]^{\left[(g|\ol{f})\;,\;\;0\right]} \\ (0 \ar[r] & P_g) \\  } \notag
\end{align}
\begin{align}
\xymatrix@R=30pt{ T\;\;=\ar@<-9pt>[d]_{\beta_{e,\ol{e}}}  \\ T\;\;=} \quad 
\xymatrix@C=70pt@R=30pt{ (P_E \ar[r]|{\left[\begin{smallmatrix}(f|e)\\(f|\ol{e})\end{smallmatrix}\right]} \ar[d]_{(e|\ol{e})} & P_f\oplus P_f)\ar[d]^{\left[\begin{smallmatrix}0&\mathrm{id}\\0&0\end{smallmatrix}\right]} \\ (P_E \ar[r]|{\left[\begin{smallmatrix}(f|e)\\ (f|\ol{e})\end{smallmatrix}\right]} & P_f\oplus P_f) } \notag 
\end{align}
\begin{align}
\xymatrix@R=35pt{ P_f \;\;=\ar@<-9pt>[d]_{\beta_{\ol{e},\ol{f}}}  \\ T\;\;=} \quad 
\xymatrix@C=70pt@R=35pt{ (0 \ar[r] \ar[d]_{0} & P_{\ol{f}})\ar[d]^{\left[\begin{smallmatrix}0\\\mathrm{id}\end{smallmatrix}\right]} \\ (P_E \ar[r]|{\left[\begin{smallmatrix}(f|e)\\(f|\ol{e})\end{smallmatrix}\right]} & P_f\oplus P_f). } \notag 
\end{align}
In the case when $\ol{f}$ is truncated in $\G$, since $\alpha_{g,\ol{f}}=\alpha_{\ol{f},\ol{f}}$ is not an arrow of $Q_\G$, we only need to add new arrows.
In the case when $f$ is truncated in $\mu_E^-(\G)$, since \cite[6.2]{A2} implies that $\beta_{f,\sigma(\ol{e})}:P_{\sigma(\ol{e})}=P_f\to P_f$ is not irreducible, we can simply ignore anything involving $\beta_{f,\sigma(\ol{e})}$ in what follows.
Nevertheless, all the computations below are valid as compositions between maps of the given form, whether they are monomial of $Q_1$ or not.

We first check that (SB2) and (SB2') hold for $\beta_{f,\sigma(\ol{e})}, \beta_{g,e}, \beta_{e,\ol{e}}, \beta_{\ol{e},\ol{f}}$.
\begin{align}
\beta_{f,\sigma(\ol{e})}: & \begin{cases}
(\sigma^{-1}(f)\,\vline\,f)\beta_{f,\sigma(\ol{e})} & \notin I; \\ \beta_{\ol{e},\ol{f}}\beta_{f,\sigma(\ol{e})} & \in I; \\
\beta_{f,\sigma(\ol{e})}\left(\sigma(\ol{e})\,\vline\,\sigma^2(\ol{e})\right) & \notin I; \\
\beta_{f,\sigma(\ol{e})}\left(\ol{\sigma(\ol{e})}\,\vline\,\sigma(\ol{\sigma(\ol{e})})\right) & \in I, \\
\end{cases} 
& \beta_{g,e}: & \begin{cases}
\left(\sigma^{-1}(g)\,\vline\,g\right)\beta_{g,e} &\notin I;\\ 
\left(\sigma^{-1}(\ol{g})\,\vline\,\ol{g}\right)\beta_{g,e} &\in I; \\
\beta_{g,e}\beta_{e,\ol{e}} &\notin I;\\
\beta_{g,e}\beta_{\ol{e},\ol{f}} &\in I,\\
\end{cases}
\notag \\
\beta_{e,\ol{e}}: & \begin{cases}
\beta_{g,e}\beta_{e,\ol{e}} &\notin I; \qquad (*)\\
\beta_{e,\ol{e}}\beta_{e,\ol{e}} &\in I;\\
\beta_{e,\ol{e}}\beta_{\ol{e},\ol{f}} &\notin I,\\
\end{cases}
& \beta_{\ol{e},\ol{f}}: & \begin{cases}
\beta_{e,\ol{e}}\beta_{\ol{e},\ol{f}} &\notin I; \qquad (*)\\ 
\beta_{g,e}\beta_{\ol{e},\ol{f}} &\in I; \qquad (*)\\
\beta_{\ol{e},\ol{f}}\left(\ol{f}\,\vline\,\sigma(\ol{f})\right) &\notin I;\\
\beta_{\ol{e},\ol{f}}\beta_{f,\sigma(\ol{e})} &\in I. \qquad (*)\\
\end{cases}\notag 
\end{align}
Note that (SB2) (resp. (SB2')) holds for $\beta$ if and only if both of the first two (resp. last two) statements in the group labelled by $\beta$ hold.
The mark $(*)$ was used to indicate the duplicated entries.

Let us start with the group labelled by $\beta_{f,\sigma(\ol{e})}$.
The first, third, and fourth conditions are easy to check as they follow directly from the Brauer relations associated to $\Lambda_\G$.
To show $\beta_{\ol{e},\ol{f}}\beta_{f,\sigma(\ol{e})} \in I$, consider the following commutative diagram:
\[
\xymatrix@R=35pt{ P_{\sigma(\ol{e})} \;\;=\ar@<-9pt>[d]_{\beta_{\ol{e},\ol{f}}\beta_{f,\sigma(\ol{e})}}  \\ T\;\;=} \quad 
\xymatrix@C=70pt@R=35pt{ (0 \ar[r] \ar[d]_{0} & P_{\sigma(\ol{e})})\ar[d]^{\left[\begin{smallmatrix} 0 \\ (f|\sigma(\ol{e}))\end{smallmatrix}\right]} \ar@{-->}[ld]|{(\ol{e}|\sigma(\ol{e}))} \\ (P_E \ar[r]|{\left[\begin{smallmatrix}(f|e)\\(f|\ol{e})\end{smallmatrix}\right]} & P_f\oplus P_f) }
\]
Then we see that the composition of maps is null-homotopy.
Hence, $\beta_{\ol{e},\ol{f}}\beta_{f,\sigma(\ol{e})}=0$ in the endomorphism algebra $\Gamma$ over the bounded homotopy category.

Consider the group labelled by $\beta_{g,e}$.
The first composition is a map concentrated only in degree 0 given by $(\sigma^{-1}(g)|\ol{f})$, so it is not null-homotopy.
The second composition is zero by the Brauer relation (Br2).
The composition $\beta_{g,e}\beta_{e,\ol{e}}$ clearly is zero in all but the $0$-th degree.
The degree 0 component is given by $[(g|\ol{f}),0]
\left[\begin{smallmatrix}
0 & \mathrm{id} \\ 0 & 0
\end{smallmatrix}\right] = \left[0,(g|\ol{f})\right] \neq 0$.
Since the degree $-1$ component of the stalk complex $P_g$ is zero, there is no map for which $(g|\ol{f})$ can factor through, meaning that the composition is not null-homotopy - as required.
The composition $\beta_{g,e}\beta_{\ol{e},\ol{f}}$ is given by the map $[(g|\ol{f}),0][0,1]^t$ concentrated in degree 0, and this is zero.

Consider the group labelled by $\beta_{e,\ol{e}}$.
The degree $-1$ component of $\beta_{e,\ol{e}}^2$ is given by $(e|\ol{e})^2=0$, and degree 0 component is given by $\left[\begin{smallmatrix}
0 & \mathrm{id} \\ 0 & 0
\end{smallmatrix}\right]^2 = 0$; hence, $\beta_{e,\ol{e}}^2=0$.
The third composition $\beta_{e,\ol{e}}\beta_{\ol{e},\ol{f}}$ is a map concentrated in degree 0 given by $\left[\begin{smallmatrix}
0 & \mathrm{id} \\ 0 & 0
\end{smallmatrix}\right]\left[\begin{smallmatrix}
0 \\ \mathrm{id}
\end{smallmatrix}\right]=\left[\begin{smallmatrix}
\mathrm{id} \\ 0
\end{smallmatrix}\right]:P_f\to P_f\oplus P_f$.
This is not null-homotopy as identity map on $P_f$ cannot factor through $P_E$.

The only remaining condition we need to check is the third row in  the $\beta_{\ol{e},\ol{f}}$-group.
The composition is determined by its degree 0 component $[0,\mathrm{id}]^t(\ol{f}\,\vline\,\sigma(\ol{f})) = [0, (\ol{f}\,\vline\,\sigma(\ol{f}))]^t : P_{\sigma(\ol{f})}\to P_f\oplus P_f$.
By the irreducibility of the arrow $(\ol{f}\,\vline\,\sigma(\ol{f}))$, which means that it cannot factor through $P_E$, we get that $\beta_{\ol{e},\ol{f}}(\ol{f}|\sigma(\ol{f}))$ is not null-homotopy.

\subsection*{Case 1b}
Recall that $E$ is a loop such that there is some $e\in E$ with $\ol{e}=\sigma(e)$.
As oppose to Case 1a, we now assume $\ol{f}=\sigma(\ol{e})$, or equivalently $g=\ol{e}$.

In this situation, the quiver $Q$ is the same as $Q_\G$ but three of the arrows are given by different morphisms.  Namely, $\alpha_{e,\ol{e}}, \alpha_{\ol{e},\ol{f}}, \alpha_{f,e}$ are replaced by $\beta_{e,\ol{e}}, \beta_{\ol{e},\ol{f}}, \beta_{f,e}$ respectively.
The definition of $\beta_{f,e}$ is similar to $\beta_{g,e}$ in the previous case - we just need to replace all the $g$'s by $f$'s.
The other two maps are given by the corresponding maps of the same notation in Case 1a.

The calculations of the required conditions are mostly the same as the last three groups in Case 1a after replacing all the $g$'s by $f$'s.
Note that (both of) the conditions $(\sigma^{-1}(\ol{g})|\ol{g})\beta_{g,e}$ and $\beta_{\ol{e},\ol{f}}\beta_{f,\sigma(\ol{e})}$ are replaced by $\beta_{\ol{e},\ol{f}}\beta_{f,e}$.
To show this composition is null-homotopy is similar to proving the condition $\beta_{\ol{e},\ol{f}}\beta_{f,\sigma(\ol{e})}$ in Case 1a:
\begin{align}
\xymatrix@R=35pt{ T \;\;=\ar@<-9pt>[d]_{\beta_{\ol{e},\ol{f}}\beta_{f,e}}  \\ T\;\;=} \quad 
\xymatrix@C=70pt@R=35pt{ (P_E \ar[r]|{\left[\begin{smallmatrix} (f|e) \\ (f|\ol{e})\end{smallmatrix}\right]} \ar[d]_{0}& P_f\oplus P_f)\ar[d]^{\left[\begin{smallmatrix} 0 & 0\\ (f|\ol{f})& 0\end{smallmatrix}\right]} \ar@{-->}[ld]|{[(\ol{e}|\ol{f}),0]} \\ (P_E \ar[r]|{\left[\begin{smallmatrix}(f|e)\\(f|\ol{e})\end{smallmatrix}\right]} & P_f\oplus P_f) } \notag
\end{align}

\subsection*{Case 2}
Suppose there is no $e\in E$ with $\ol{e}=\sigma(e)$.
The new summand in $\mu_E^-(\Lambda_G)$ is $T:= (P_E\xto{d} P_{f}\oplus P_{f'})$ where $d=\left[(f|e)\;,\;(f'|\ol{e})\right]$.
Note that if $e$ (resp. $\ol{e}$) is truncated, then we will just remove $P_f$ and $\alpha_{f,e}$ (resp. $P_{f'}$ and $\alpha_{f',\ol{e}}$).
Moreover, in such a case, one should ignore the respective entries and diagrams in what follows.

\subsection*{Case 2a}
Assume further that neither $f'=\ol{\sigma(\ol{e})}$ (equivalently, $g'=\ol{e}$) nor $f=\ol{\sigma(e)}$ (equivalently, $g=e$).

The new quiver $Q$ is
\[
\left(Q_\G \setminus \{ \alpha_{e,\sigma(e)},\alpha_{f,e},\alpha_{g,\ol{f}}, \alpha_{\ol{e},\sigma(\ol{e})},\alpha_{f',\ol{e}},\alpha_{g',\ol{f'}} \}\right) \cup \left\{
\begin{array}{rrr} \beta_{f,\sigma(e)}, & \beta_{e,\ol{f}}, & \beta_{g,e}, \\ \beta_{f',\sigma(\ol{e})},& \beta_{\ol{e},\ol{f'}},& \beta_{g',\ol{e}}\end{array}\right\},
\]
where the new arrows are given by
\begin{align}
\beta_{f,\sigma(e)} := \left(\alpha_{f,\sigma(e)}:P_{\sigma(e)} \to P_{f}\right)
\hspace*{2cm}
&
\beta_{f',\sigma(\ol{e})} := \left( \alpha_{f',\sigma(\ol{e})}:P_{\sigma(\ol{e})} \to P_{f'}\right) \hspace*{2cm} 
\notag 
\end{align}
\begin{align}
\xymatrix@R=35pt{ P_f \;\;=\ar@<-9pt>[d]_{\beta_{e,\ol{f}}}  \\ T\;\;=} \quad 
\xymatrix@C=70pt@R=35pt{ (0 \ar[r]\ar[d]_{0} & P_f\ar[d]_{\left[\begin{smallmatrix}\mathrm{id}\\0\end{smallmatrix}\right]}) \\ (P_E \ar[r]|{\left[\begin{smallmatrix}(f|e)\\(f'|\ol{e})\end{smallmatrix}\right]} & P_f\oplus P_{f'}) \\  }
&
\xymatrix@R=35pt{ P_{f'} \;\;=\ar@<-9pt>[d]_{\beta_{\ol{e},\ol{f'}}}  \\ T\;\;=} \quad 
\xymatrix@C=70pt@R=35pt{ (0 \ar[r]\ar[d]_{0} & P_{f'}\ar[d]_{\left[\begin{smallmatrix}0\\\mathrm{id}\end{smallmatrix}\right]}) \\ (P_E \ar[r]|{\left[\begin{smallmatrix}(f|e)\\(f'|\ol{e})\end{smallmatrix}\right]} & P_f\oplus P_{f'}) \\  } \notag \\
& \notag 
\end{align}
\begin{align}
\xymatrix@R=35pt{ T \;\;=\ar@<-9pt>[d]_{\beta_{g,e}}  \\ P_g\;\;=} \quad 
\xymatrix@C=70pt@R=35pt{ (P_E \ar[r]|{\left[\begin{smallmatrix}(f|e)\\(f'|\ol{e})\end{smallmatrix}\right]} \ar[d]_{0} & P_f\oplus P_{f'})\ar[d]|{\left[(g|\ol{f})\;,\;0\right]\quad} \\ (0 \ar[r] & P_g) \\  } 
&
\xymatrix@R=35pt{ T \;\;=\ar@<-9pt>[d]_{\beta_{g',\ol{e}}}  \\ P_{g'}\;\;=} \quad 
\xymatrix@C=70pt@R=35pt{ (P_E \ar[r]|{\left[\begin{smallmatrix}(f|e)\\(f'|\ol{e})\end{smallmatrix}\right]} \ar[d]_{0} & P_f\oplus P_{f'})\ar[d]|{\left[0\;,\;(g'|\ol{f'})\right]} \\ (0 \ar[r] & P_g). \\  } \notag
\end{align}
Note that, analogous to Case 1a, if $\ol{f}$ (resp. $\ol{f'}$) is truncated in $\G$, then we ignore $\alpha_{g,\ol{f}}=\alpha_{\ol{f},\ol{f}}$ (resp. $\alpha_{g',\ol{f'}}$) as an arrow of $Q_\G$, whereas $(g|\ol{f})$ in the definition of $\beta_{g,e}$ (resp. $(g'|\ol{f'})$ in the definition of $\beta_{g',\ol{e}}$) becomes the Brauer cycle $(\ol{f}|\ol{f})$ (resp. $(\ol{f'}|\ol{f'})$).
Similarly, if $f$ (resp. $f'$) is truncated in $\mu_E^-(\G)$, then we can ignore any computations involving $\beta_{f,\sigma(e)}$ (resp. $\beta_{f',\sigma(\ol{e})}$).

The following shows half of the list of conditions one needs to check - the other half can be obtained by swapping $e,f,g$ with $\ol{e},f',g'$ respectively.
\[
\begin{array}{ll|ll|ll}
\beta_{f,\sigma(e)} & & \beta_{e,\ol{f}} & & \beta_{g,e} & \\ 
\hline
\left(\sigma^{-1}(f)\,\vline\,f\right)\beta_{f,\sigma(e)} & \notin I & \beta_{g,e}\beta_{e,\ol{f}} &\notin I &
\left(\sigma^{-1}(g)\,\vline\,g\right)\beta_{g,e} &\notin I \rule{0pt}{1.5em} \\  
\beta_{e,\ol{f}}\beta_{f,\sigma(e)} & \in I & \beta_{g',\ol{e}}\beta_{e,\ol{f}} &\in I &
\left(\sigma^{-1}(\ol{g})\,\vline\,\ol{g}\right)\beta_{g,e} &\in I \\
\beta_{f,\sigma(e)}\left(\sigma(e)\,\vline\,\sigma^2(e)\right) & \notin I & \beta_{e,\ol{f}}\left(\ol{f}\,\vline\,\sigma(\ol{f})\right) &\notin I & \beta_{g,e}\beta_{e,\ol{f}} &\notin I \quad (*)\\
\beta_{f,\sigma(e)}\left(\ol{\sigma(e)}\,\vline\,\sigma(\ol{\sigma(e)})\right) & \in I &
\beta_{e,\ol{f}}\beta_{f,\sigma(e)} &\in I \quad (*) & \beta_{g,e}\beta_{\ol{e},\ol{f'}} &\in I.
\end{array}
\]
Checking conditions in the group labelled by $\beta_{f,\sigma(e)}$ (resp. $\beta_{e,\ol{f}}$, resp. $\beta_{g,e}$) is similar to the the group labelled by $\beta_{f,\sigma(\ol{e})}$ (resp. $\beta_{e,\ol{f}}$, resp. $\beta_{g,e}$) in Case 1a by suitably replacing half-edges.
The same applies to the three other groups of relations associated to $\ol{e}$.

\subsection*{Case 2b}
Now there is no $e\in E$ with $\ol{e}=\sigma(e)$
and (i) $f=\ol{\sigma(e)}$ (equivalently, $g=e$) or (ii) $f'=\ol{\sigma(\ol{e})}$ (equivalently, $g'=\ol{e}$) hold(s).
Without loss of generality, we assume that (i) holds; in the case if (ii) also hold, one just applies the following arguments again after replacing $e,f,g$ by $\ol{e},f',g'$ respectively.

Under this assumption, we will formally remove $\{\alpha_{e,\ol{f}}, \alpha_{f,e}\}$ and replace by $\{\beta_{e,\ol{f}}, \beta_{f,e}\}$.
Roughly speaking, the difference between Case 2a and Case 2b is analogous to the difference between Case 1a and Case 1b.
In particular, the modifications needed to show (SB2) and (SB2') holds in the new arrows are analogous to those in Case 1b - anything labelled by $g$ in Case 2a should now be labelled by $f$, and the relations $\beta_{e,\ol{f}}\beta_{f,\sigma(e)}=0$ and $(\sigma^{-1}(\ol{g})|\ol{g})\beta_{g,e}=0$ are replaced by ``$\beta_{e,\ol{f}}\beta_{f,e}$ is null-homotopy".
We leave these minor details as exercise for the reader.

\medskip

Finally, the (SB2) and (SB2') conditions for all other arrows (i.e. those labelled by $\alpha_{x,y}$ in $Q$) are inherited from $\Lambda_\G$ or follow from one of the conditions listed above.
It follows that $\Gamma$ indeed satisfies (SB2) and (SB2').
This completes the proof of Lemma \ref{lem-flip-alg}. \qed

\pagebreak

\section{List of notations}\label{sec-notations}
\begin{tabbing}
$z\subset w$\hspace{2.8cm}\=\parbox{4.8in}{non-empty continuous subsequence\dotfill\pageref{sym-subseq}}\\
\addsym $S_i$: {simple module corresponding to $i$}{sym-simple}
\addsym $P_i$: {indecomposable projective module corresponding to $i$}{sym-proj}
\addsym $|X|$: {number of isoclasses of the indecomposable summands of $X$}{sym-cardinal}
\addsym $\G$: {ribbon graph or Brauer graph}{sym-G}
\addsym $V$: {set of vertices}{sym-V}
\addsym $H$: {set of half-edges}{sym-H}
\addsym $s$: {emanating vertex specifier}{sym-s}
\addsym $\ol{e}$: {involution acting on $e$}{sym-ol}
\addsym $\sigma$: {cyclic ordering}{sym-sigma}
\addsym $(e,\sigma(e),\ldots)_v$: {cyclic ordering around $v$}{sym-cyclic-order-v}
\addsym $\val$: {valency}{sym-val}
\addsym $\mm$: {multiplicity function}{sym-mm}
\addsym $w=(e_1,\ldots,e_l)$: {half-walk}{sym-halfwalk}
\addsym $W$: {walk}{sym-walk}
\addsym $\e_W$: {signature}{sym-signature}
\addsym $(w;\e)$ or $(e_1^+,e_2^-,\ldots)$: {signed half-walk}{sym-signed-hw}
\addsym $\SW(\G)$: {set of signed walks}{sym-SWG}
\addsym $w\cap w'$: {set of maximal continuous subsequences common in $w$ and $w'$}{sym-common-subwalk}
\addsym $W\cap W'$: {set of maximal common subwalks of $W$ and $W'$}{sym-set-common-subwalk}
\addsym $\vir_\pm(e)$: {virtual edges associated to $e$}{sym-vir}
\addsym $\AW(\G)$: {set of admissible walks}{sym-AWG}
\addsym $\tilt\Lambda$: {set of tilting complexes}{sym-tilt}
\addsym $T\geq U$: {a partial order on tilting complexes}{sym-TgeqU}
\addsym $\ntilt{n}\Lambda$: {set of $n$-term tilting complexes}{sym-ntilt}
\addsym $\ipt\Lambda$: {set of indecomposable two-term pretilting complexes}{sym-ipt}
\addsym $\Lambda_\G$: {Brauer graph algebra associated to a Bruaer graph $\G$}{sym-LambdaG}
\addsym $Q_\G$: {quiver associated to a Brauer graph $\G$}{sym-QG}
\addsym $1_E, [e]^0$: {idempotent at $E=\{e,\ol{e}\}$}{sym-e0}
\addsym $(e|\sigma(e))$: {an arrow in $Q_\G$, or irreducible map between projectives}{sym-a-e-sig-e}
\addsym $(e|f)$: {a short path in $Q_\G$, or a short map}{sym-a-ef}
\addsym $[e]$: {a Brauer cycle, i.e. shorthand for $(e|e)$}{sym-a-ee}
\addsym $P_M$: {minimal projective presentation of $M$}{sym-PM}
\addsym $d_M$: {differential map in $P_M$}{sym-dM}
\addsym $\gcx\Lambda$: {set of certain type of short string complexes and stalk projectives}{sym-gcx}
\addsym $T_w, T_W$: {two-term complex associated to a (half-)walk ($w$ or) $W$}{sym-Tw}
\addsym $\CW(\G)$: {set of complete admissible sets of signed walks on $\G$}{sym-CWG}
\addsym $M_T$: {zeroth cohomology of the complex $T$}{sym-MT}
\addsym $N_T$: {$(-1)$-st cohomology of the complex $T$}{sym-NT}
\addsym $A$: {$\Lambda_\G/\soc\Lambda_\G$}{sym-A}
\addsym $\hd(\alpha)$: {head of an arrow or a word}{sym-head}
\addsym $\tail(\alpha)$: {tail of an arrow or a word}{sym-tail}
\addsym $\mathcal{A}$: {set of alphabets associated to $A$}{sym-alphabet-set}
\addsym $\alpha^{-1}$: {formal inverse of an arrow or a word}{sym-inverse}
\addsym $\word, \mathsf{u}$: {word or string of $A$}{sym-word}
\addsym $Q_\word$: {quiver associated to a string $\word$ of $A$}{sym-Qword}
\addsym $\eta_e$: {longest path in the hook module}{sym-hook}
\addsym $\mu_X^-(T),\mu_X^+(T)$: {mutation of the tilting complex $T$ at direct summand $X$}{sym-muXT}
\addsym $\T_\Lambda$: {Hasse quiver of $\tilt(\Lambda)$}{sym-TLambda}
\addsym $\mu_E^-(\G),\mu_E^+(\G)$: {mutation of Brauer tree}{sym-muEG}
\addsym $\G^\op$: {the opposite of $\G$}{sym-Gop}

\end{tabbing}


\end{document}